\documentclass[a4paper,12pt,reqno]{amsart}

\usepackage{amsmath,amssymb,latexsym,amstext,enumerate}
\usepackage{dsfont}

\theoremstyle{definition}
\newtheorem{defi}{Definition}[section]

\theoremstyle{plain}

\newtheorem{thm}[defi]{Theorem}

\newtheorem{lem}[defi]{Lemma}
\newtheorem{prop}[defi]{Proposition}
\newtheorem*{prop*} {Proposition}
\newtheorem{cor}[defi]{Corollary}
\theoremstyle{remark}
\newtheorem{rmk}[defi]{Remark}

\usepackage{setspace}

\usepackage{bm}

\numberwithin{equation}{section}

\newcommand{\1}{ \mathds{1}}
\newcommand{\comm}[1]{ }

\DeclareMathOperator{\Aut}{Aut}
\DeclareMathOperator{\End}{End}
\DeclareMathOperator{\wt}{wt}

\newcommand{\Z}{\mathbb{Z}}
\newcommand{\M}{\mathbb{M}}
\newcommand{\R}{\mathbb{R}}
\newcommand{\la}{\langle}
\newcommand{\ra}{\rangle}
\newcommand{\C}{\mathbb{C}}

\newcommand{\nop}{\mbox{$\circ\atop\circ$}}
\newcommand{\h}{\mathfrak{h}}

\newcommand{\F}{ \mathbb{F}}

\newcommand{\cC}{ \bm{\mathcal{C}}} 
\newcommand{\cD}{ \bm{\mathcal{D}}} 

\newcommand{\LCD}{{L_{ \cC \times \cD}} } 

\DeclareMathOperator{\Suz}{Suz}

\DeclareMathOperator{\GL}{GL}

\DeclareMathOperator{\Ker}{Ker}
\DeclareMathOperator{\Hom}{Hom}

\DeclareMathOperator{\Span}{Span}

\begin{document}

\title[$\Z_3$-orbifold  construction of the Moonshine VOA]{$\Z_3$-orbifold  construction of the Moonshine vertex operator algebra and some 
maximal $3$-local subgroups of the Monster}

 \author{Hsian-Yang Chen}
  \address[H.Y.  Chen]{ National University of Tainan, Tainan  70005, Taiwan}
\email{hychen@mail.nutn.edu.tw}
 \author{Ching Hung Lam}
 \address[C.H. Lam]{ Institute of Mathematics, Academia Sinica, Taipei  10617, Taiwan
 and National Center for Theoretical Sciences, Taiwan}
\email{chlam@math.sinica.edu.tw}

\author{Hiroki Shimakura}%
\address[H. Shimakura]{Graduate School of Information Sciences,
Tohoku University,
Sendai 980-8579, Japan }%
\email {shimakura@m.tohoku.ac.jp}
\thanks{C.\,H. Lam was partially supported by MoST grant 104-2115-M-001-004-MY3 of Taiwan}
\thanks{H.\ Shimakura was partially supported by JSPS KAKENHI Grant Numbers 26800001}
\thanks{C.\,H. Lam and H.\ Shimakura were partially supported by JSPS Program for Advancing Strategic International Networks to Accelerate the Circulation of Talented Researchers ``Development of Concentrated Mathematical Center Linking to Wisdom of the Next Generation".}



\begin{abstract}
In this article, we describe some maximal $3$-local subgroups of the Monster simple group using vertex operator algebras (VOA). 
We first study the holomorphic vertex operator algebra
obtained by applying the orbifold construction to the Leech lattice vertex operator algebra and a lift of a fixed-point free isometry of order $3$ of the Leech lattice. 
We also consider some of its special subVOAs and study their stabilizer subgroups using the symmetries of the subVOAs. It turns out that these stabilizer subgroups are $3$-local subgroups of its full automorphism group.
As one of our main results,  
we show that its full automorphism group is isomorphic to the Monster simple group by using a $3$-local characterization and that the holomorphic VOA is isomorphic to the Moonshine VOA. This approach allows us to obtain relatively explicit descriptions of two 
maximal $3$-local subgroups of the shape $3^{1+12}.2.\Suz{:}2$ and $3^8.\Omega^-(8,3).2$ in the Monster simple group. 
\end{abstract}
\maketitle


\section{Introduction}

The Moonshine vertex operator algebra (VOA) $V^\natural$ is one of the most interesting VOAs.
Its full automorphism group $\Aut V^\natural$ is isomorphic to the Monster simple group $\M$, the largest sporadic finite simple group. This view point is useful for understanding certain aspects of $\M$.
For example, $V^\natural$ was constructed in \cite{FLM} from the lattice VOA $V_\Lambda$ associated to the Leech lattice $\Lambda$ and an automorphism $\theta\in\Aut V_\Lambda$ of order $2$, which is  a lift of the $-1$ isometry of $\Lambda$, as follows:
$$V^\natural:=V_\Lambda^+\oplus V_\Lambda^{T,+},$$ where $V_\Lambda^+$ is the $\theta$-fixed-point subspace of $V_\Lambda$ and $V_\Lambda^{T,+}$ is the subspace of the (unique) irreducible $\theta$-twisted module $V_\Lambda^T$ with integral weights.
By the construction, there is a natural automorphism of order $2$ that acts as $1$ on $V_\Lambda^+$ and as $-1$ on $V_\Lambda^{T,+}$. The centralizer of this automorphism coincides with the stabilizer in $\Aut V^\natural (\cong \M)$ of the subVOA $V_\Lambda^{+}$ and it has the shape $2^{1+24}_+.Co_1$ (cf. \cite{FLM,Sh07}).
In addition, four other maximal $2$-local subgroups of $\M$ are described as the stabilizers in $\Aut V^\natural$ of certain subVOAs of $V_\Lambda^+$ (\cite{Sh07}).
However, it is not easy to describe other maximal subgroups based on this construction. 
In order to study the other $p$-local subgroups of the Monster,  it is natural to ask if there exists a similar $\Z_p$-orbifold construction of $V^\natural$ associated to $V_\Lambda$ and an automorphism of order $p$ and if there is a natural way to describe maximal $p$-local subgroups using such a construction. In this article, we will consider the case where $p=3$ and try to obtain explicit descriptions of some maximal $3$-local subgroups of the Monster. 

Let  $\tau\in\Aut V_\Lambda$  be a lift of a fixed-point free isometry of $\Lambda$ of order $3$ such that $\tau$ has order $3$.
It is well-known \cite{DLM00} that $ V_ \Lambda$ has a (unique) irreducible $ \tau^i$-twisted module for $ i = 1,2$; We denote these twisted modules by $V_\Lambda ^{T_1}(\tau)$ and $ V_ \Lambda^{T_2}(\tau^2)$, respectively.
We also define the $ V_\Lambda^\tau$-module
\begin{align*}
 V^ \sharp :=&   V_ \Lambda^\tau \oplus  (V_ \Lambda ^{T_1}(\tau))_ \mathbb{Z} \oplus  (V_ \Lambda ^{T_2}(\tau^2))_ \mathbb{Z},
\end{align*}
where $V_\Lambda^\tau$ is the $\tau$-fixed point subspace of $V_\Lambda$ and $ (V_\Lambda ^{T_i}(\tau^i))_ \mathbb{Z} $ is the subspace of $ V_\Lambda ^{T_i}(\tau^i)$ with integral weights.

It was first announced by Dong and Mason~\cite{DM94p} that $V^\sharp$ has a natural VOA structure as an extension of the $V_\Lambda^\tau$-module structure. They also claimed that the
full automorphism group of $V^\sharp$ is isomorphic to $\M$ and $V^\sharp$ is isomorphic to  $V^\natural$ as a VOA. However, a
complete proof has not been published.
Recently, Miyamoto proved in  \cite{Mi13a} that  $V^\sharp$ has a holomorphic VOA structure (see also \cite{TY13}). He also showed that $V^\sharp$ is a $\mathbb{Z}_3$-graded simple current extension of the VOA $V_\Lambda^\tau$. 

In this article, we describe two maximal $3$-local subgroups of the Monster simple group $\M$ as lifts of symmetries of certain subVOAs of the Moonshine VOA by the following way.
We view $V^\sharp$ as simple current extensions of the full subVOAs $V_\Lambda^\tau$ and $(V_{K_{12}}^\tau)^{\otimes2}$ graded by $\Z_3$ and $\Z_3^8$, respectively, where $K_{12}$ is the Coxeter-Todd lattice of rank $12$.
By using the fusion rules of these subVOAs, we show that the automorphism groups of $V_\Lambda^\tau$ and $(V_{K_{12}}^\tau)^{\otimes2}$ have the shapes $3^{12}.2\Suz{:}2$ and $(\Omega^-_8(3){:}2)\wr\Z_2$, respectively.
By a lifting theory in simple current extensions, we show that the normalizers in $\Aut V^\sharp$ of the corresponding elementary abelian $3$-subgroups of order $3$ and $3^8$ have the shapes $3^{1+12}.2\Suz{:}2$ and $3^8. \Omega^-_8(3).2$, respectively.
Indeed, they are the stabilizers of the subVOAs $V_\Lambda^\tau$ and $(V_{K_{12}}^\tau)^{\otimes2}$ in $\Aut V^\sharp$, respectively.
By using the $3$-local characterization of $\M$, we see that $\Aut V^\sharp$ is isomorphic to $\M$.
In particular, these two $3$-local subgroups are maximal in $\M$.

The action of $\M$ on a holomorphic VOA of central charge $24$ with trivial weight one space is not enough to characterize the Moonshine VOA.
Indeed, by the character theory of finite groups, a $\M$-invariant $(1+196883)$-dimensional Griess algebra structure has one parameter, up to isomorphism.
We show that it is specified by the additional conditions, the unitary structure of a VOA and the existence of an Ising vector.
By checking that $V^\sharp$ satisfies the conditions and applying the result in \cite{DGL}, we obtain the following theorem:

\begin{thm}\label{Thm:main}
The holomorphic vertex operator algebra $V^\sharp$ is isomorphic to the Moonshine vertex operator algebra $V^\natural$.
\end{thm}

Let us explain our methods in more detail.

First, we describe two subgroups $H_1$ and $H_2$ of $\Aut V^\sharp$ of shape $3^{1+13}.2\Suz{:}2$ and $3^8.\Omega^-_8(2).2$, respectively.
These subgroups are, in fact, maximal $3$-local in $\M$.
The main idea is to use the fusion rules of the orbifold VOAs $V_\Lambda^\tau$ and $V_{K_{12}}^\tau$, where $K_{12}$ denotes the Coxeter-Todd lattice of rank $12$ and $\tau$ also denotes  a fixed-point free automorphism of $K_{12}$ of order $3$. 
It was shown in \cite{CL14} that all irreducible modules of the VOAs $V_\Lambda^\tau$ and $V_{K_{12}}^\tau$ are simple currents.

To describe the subgroup $H_1$, we
consider the action of $\Aut V_\Lambda^\tau$ on the set of isomorphism classes of irreducible $V_\Lambda^\tau$-modules and study the lift of automorphisms of $V_\Lambda^\tau$ to those of $V_\Lambda$ or $V^\sharp$. By using \cite{Sh04}, we show that $\Aut V_\Lambda^\tau$ is isomorphic to the quotient of the normalizer of $\la \tau\ra $ in $\Aut V_\Lambda$ by $\langle\tau\rangle$, and the stabilizer $H_1$ in $\Aut V^\sharp$ of $V_\Lambda^\tau$ is isomorphic to an extension of $\Aut V_\Lambda^\tau$ by $\Z_3$.
Thus $H_1$ has the shape $3^{1+12}.2\Suz{:}2$.

For the subgroup $H_2$, we notice that the set of (inequivalent) irreducible $V_{K_{12}}^\tau$-modules $R=R(V_{K_{12}}^\tau)$ forms an elementary abelian $3$-group of order $3^8$ under the fusion rules. 
In fact, it forms an $8$-dimensional non-singular quadratic space over $\F_3$ of minus type with respect to the quadratic form defined by lowest weights (see Proposition \ref{thm:O83minus}).
By using the action of $\Aut V_{K_{12}}^\tau$ on $(R,q)$ as a subgroup of the orthogonal group $O(R,q)$, we show in Proposition \ref{Prop:AutVK12tau} that $\Aut V_{K_{12}}^\tau$ has the shape $\Omega^-_8(3){:}2$ (cf.\ \cite{LY13}).
It is well-known that the Leech lattice $\Lambda$ contains $K_{12}\perp K_{12}$ as a full sublattice (e.g. \cite{KLY03}).  
Hence $V_{\Lambda}^\tau$ contains $(V_{K_{12}}^\tau)^{\otimes 2}$ as a full subVOA, and so does $V^\sharp$. 
By viewing $V^\sharp$ as a $\Z_3^8$-graded simple current extension of $(V_{K_{12}}^\tau)^{\otimes 2}$, we obtain an $8$-dimensional maximal totally singular subspace $S^\sharp$ of $R^2(=R\oplus R)$ corresponding to $V^\sharp$.
By \cite{Sh04} (cf.\ \cite{Sh07}), the stabilizer of $(V_{K_{12}}^\tau)^{\otimes 2}$ in $\Aut V^\sharp$ is an extension of the stabilizer of $S^\sharp$ in $\Aut (V_{K_{12}}^\tau)^{\otimes 2}\cong \Aut V_{K_{12}}^\tau\wr\Z_2$ by $\Z_3^8$.
By direct computation, we show that the stabilizer of $S^\sharp$ has the shape $\Omega^-_8(3).2$.
Hence the stabilizer $H_2$ of $(V_{K_{12}}^\tau)^{\otimes 2}$ in $\Aut V^\sharp$ has the shape $3^8. \Omega^-_8(3).2$.

We then show the finiteness of $\Aut V^\sharp$ by using these two subgroups and an argument of \cite{T84} (cf.\ \cite{Sh14}).
Applying the characterization of $\M$ of \cite{SS10} as a finite group containing certain $3$-local subgroups, we prove that the full automorphism group of $V^\sharp$ is isomorphic to the Monster simple group $\M$. 


\medskip

Next, we prove that $V^\sharp$ is unitary in more general setting.
Let $V$ be a self-dual, simple VOA of CFT-type. Assume that $V$ has an automorphism $f$ of order $3$ such that the fixed-point subspace $V^0$ of $f$ has a unitary VOA structure and $V$ has a unitary $V^0$-module structure.
Suppose $V$ has another automorphism $g$ of order $3$ commuting with $f$ and  there exists an automorphism $\psi$ of $V$ such that $\psi^{-1}f\psi=g$ and $\psi^{-1}g\psi=f$.
Furthermore, we assume that the automorphism groups of $V^0$ and $V^{0,0}(=V^{\langle f,g\rangle})$ are compact.
Then we show in Theorem \ref{Thm:uni} that $V$ has a unitary VOA structure as an extension of the unitary $V^0$-module structure.
The key of the proof is that $\psi$ adjusts the unitary forms on the irreducible $V^0$-submodules of $V$.
Then we can define an anti-linear automorphism of $V$ associated to the unitary form, and check that $V$ is unitary by straightforward calculations. 

To prove that $V^\sharp$ is unitary, we check that $V^\sharp$ satisfies all the conditions in the theorem.
It was shown in \cite{Dlin} that the lattice VOA $V_\Lambda$ has a unitary VOA structure, and  so does $V_\Lambda^\tau$. Moreover, one can define unitary module structures on $V_{\Lambda}^{T_i}(\tau^i)$, $i=1,2$, by using the explicit construction of the twisted irreducible modules over a lattice VOA in \cite{Lep85,Dl96}.
Hence $V^\sharp$ is a unitary $V_\Lambda^\tau$-module.
We choose $f$ as the automorphism of $V^\sharp$ that acts on $V_\Lambda^\tau$, $(V_\Lambda^{T_1}(\tau))_\Z$, and $(V_\Lambda^{T_2}(\tau^2))_\Z$ by $1$, $\xi$ and $\xi^2$, respectively, where $\xi=\exp(2\pi\sqrt{-1}/3)$.
We can choose a suitable automorphism $g$ of order $3$ in $O_3(H_2)$ so that $(V_\Lambda^{\tau})^g$ is $V_L^\tau$ for some sublattice $L$ of $\Lambda$, where $H_2$ is the subgroup of $\Aut V^\sharp$ we have already described.
Since $f$ and $g$ belong to $O_3(H_2)$ and they correspond to singular vectors in $(R,q)$, there exists an automorphism $\psi\in H_2$ that exchanges $f$ and $g$ by conjugation.
We know that the automorphism group of $V^0(=V_\Lambda^\tau)$ is finite.
By $V^{0,0}=V_L^\tau$, we can show that the automorphism group of $V^{0,0}$ is also finite by a similar argument.
Hence $V^\sharp$ has a unitary VOA structure.

\medskip

Finally, we give a characterization of the Moonshine VOA $V^\natural$:
If $V$ is  a strongly regular, holomorphic VOA of central charge $24$ satisfying the following: (a) $V_1=0$; (b) $V$ is a unitary VOA; (c) $\Aut V$ is isomorphic to $\M$; (d) the action of $\Aut V$ on $V_2$ is non-trivial; (e) $V$ contains an Ising vector compatible with the unitary structure of $V$,
Then $V$ is isomorphic to $V^\natural$.

For such a $V$, we prove that the map $e\mapsto\tau_e$ is bijective from the set of Ising vectors in $V$ to the set of $2A$-elements of $\M$ (see Lemma \ref{Lem:1to1}), where $\tau_e$ is the $\tau$-involution associated to $e$ \cite{Mi96}.
The main task is to prove the injectivity and the key observation is that 
%
Ising vectors defining the same $\tau$-involution must be equal or orthogonal to  each other \cite{Mi96}. Note that the assumption (b) is required in \cite{Mi96}.
By the assumptions (c) and (d), the Griess algebra $V_2$ is decomposed into a direct sum of two irreducible $\M$-modules of dimension $1$ and $196883$.
Here the $1$-dimensional submodule is spanned by the conformal vector $\omega$ of $V$ and the $196883$-dimensional irreducible module is the orthogonal complement $\omega^\perp$.
For an Ising vector $e$, let $I_e$ denote the set of all Ising vectors in $V_2$ whose $\tau$-involutions coincide with $\tau_e$.
Note that the subspace $\Span_\C I_e $ is preserved by $C_\M(\tau_e)$ and its dimension is at most $48$ by the orthogonality of $I_e$.
By the possible dimensions of irreducible $C_\M(t)$-submodules of $\omega^\perp$ for order $2$ element $t$ in $\M$, $\tau_e$ must belong to the conjugacy class $2A$.
In addition, $\Span_\C I_e $ is $1$-dimensional, that is, $I_e=\{e\}$.
Hence we obtain the injectivity.

By the inner product of irreducible characters of $\M$, the commutative algebra structure on $V_2$ with an invariant form is unique up to scalar on every irreducible $\M$-submodule of $V_2$.
By the one-to-one correspondence, the power of the Ising vector in $V_2$ determines the scalar, and hence $V_2$ is isomorphic to the Monstrous Griess algebra $V^\natural_2$ as algebras. 
Our characterization of $V^\natural$ follows from \cite{DGL}.

The remaining task is to check that $V^\sharp$ satisfies (a)--(e).
Clearly, $V^\sharp$ satisfies (a)--(d). By an explicit embedding of $\sqrt{2}E_8^3$ into the Leech lattice $\Lambda$, it is easy to show that $V_\Lambda^\tau$ contains $V_{\sqrt{2}E_8}^\tau$ as a subVOA. Moreover, all Ising vectors in  $V_{\sqrt2E_8}^\theta$ are described in \cite{LSY} and they are compatible with the unitary structure of $V_{\sqrt{2}E_8}$, where $\theta$ is a lift of the $-1$-isometry.  
Since $\tau$ commutes with $\theta$, it is straightforward to show  $V_{\sqrt2E_8}^{\langle\tau,\theta\rangle}$ contains an Ising vector.
Hence, we obtain an Ising vector in $V_{\sqrt2E_8}^{\langle\tau,\theta\rangle} \subset V_{\sqrt2E_8}^{\tau}\subset V_{\Lambda}^\tau$, which proves that $V^\sharp$ satisfies (e).
Therefore, $V^\sharp$ is isomorphic to the Moonshine VOA $V^\natural$.

\begin{center}
{\bf Notations}
\begin{small}
\begin{tabular}{ll}\\
$\langle\cdot,\cdot\rangle$& the (normalized) invariant bilinear form on a VOA.\\
$a_{n}$& the $n$-th mode of an element $a\in V$.\\
$\Aut V$& the automorphism group of a VOA $V$.\\
$C_G(H)$& the centralizer of a subgroup $H$ in a group $G$.\\
$G.H$& a group extension with normal subgroup $G$ such that the quotient by $G$ is $H$.\\
$K_{12}$& the Coxeter-Todd lattice of rank $12$.\\
$\M$& the Monster simple group.\\
$N_G(H)$& the normalizer of a subgroup $H$ in a group $G$.\\
$\Omega^-_8(3)$& the commutator subgroup of the orthogonal group of\\
& $8$-dimensional quadratic space over $\F_3$ of minus type.\\
$O_3(G)$& the maximal normal $3$-subgroup of a finite group $G$.\\
$O(L)$ & the automorphism group of $L$ preserving the inner product $\langle\cdot|\cdot\rangle$.\\
$O(R,q)$& the orthogonal group of the quadratic space $R$ with quadratic form $q$.\\
$\tau$& a fixed-point free automorphism of a lattice of order $3$\\
& or its lift to an automorphism of the lattice VOA of order $3$.\\
$\Lambda$& the Leech lattice.\\
$R(V)$& the set of isomorphism classes of irreducible $V$-modules.\\
$V_L$& the lattice VOA associated to an even lattice $L$.\\
$V_\Lambda^\tau$& the fixed-point subspace of $\tau$ in $V_\Lambda$, a subVOA of $V_\Lambda$.\\
$V^\natural$& the Moonshine VOA.\\
$V^\sharp$& the holomorphic VOA of central charge $24$ constructed as in \eqref{Eq:Vsharp}.\\
$\wt(M)$ & the lowest $L(0)$-weight of an irreducible ($g$-twisted) module $M$.\\
$\omega$& the conformal vector of a VOA.\\
$W$& the $\tau$-fixed-point subspace $V_{K_{12}}^\tau$ of the Lattice VOA $V_{K_{12}}$ associated to $K_{12}$\\
$\xi$& $\exp(2\pi\sqrt{-1}/3)$.\\
$Z(G)$& the center of a finite group $G$.\\
$\Z_n$& the cyclic group of order $n$.\\
\end{tabular}
\end{small}
\end{center}

\section{Preliminary}
In this section, we will review some fundamental results about VOAs.
In addition, we review the simple current extension and related automorphism groups.

\subsection{Vertex operator algebras}
Throughout this article, all VOAs are defined over the field $\C$ of complex numbers. 
We recall the notion of vertex operator algebras (VOAs) and modules from \cite{Bo,FLM}.

A {\it vertex operator algebra} (VOA) $(V,Y,\1,\omega)$ is a $\Z_{\ge0}$-graded
 vector space $V=\bigoplus_{m=0}^\infty V_m$ equipped with a linear map
$$Y(a,z)=\sum_{i\in\Z}a_{i}z^{-i-1}\in (\End V)[[z,z^{-1}]],\quad a\in V$$
and the {\it vacuum vector} $\1$ and the {\it conformal vector} $\omega$
satisfying a number of conditions (\cite{Bo,FLM}). We often denote it by $V$.
Note that $L(n)=\omega_{n+1}$ satisfy the Virasoro relation:
$$[L{(m)},L{(n)}]=(m-n)L{(m+n)}+\frac{1}{12}(m^3-m)\delta_{m+n,0}c\ {\rm id}_V,$$
where $c$ is a complex number, called the {\it central charge} of $V$.

A linear automorphism of $V$ is called {\it an automorphism} of $V$ if it satisfies $$ g\omega=\omega\quad {\rm and}\quad gY(v,z)=Y(gv,z)g\quad
\text{ for all } v\in V.$$
We denote the group of all automorphisms of $V$ by $\Aut V$.
A {\it vertex operator subalgebra} (or a {\it subVOA}) is a graded subspace of
$V$ which has a structure of a VOA such that the operations and its grading
agree with the restriction of those of $V$ and that they share the vacuum vector.
When they also share the conformal vector, we call it a {\it full subVOA}.
For an automorphism $g$ of a VOA $V$, let $V^g$ denote the set of fixed-points of $g$.
Clearly $V^g$ is a full subVOA of $V$.

For $g\in\Aut V$ of order $p$, a $g$-twisted $V$-module $(M,Y_M)$ is a $\C$-graded vector space $M=\bigoplus_{m\in\C} M_{m}$ equipped with a linear map
$$Y_M(a,z)=\sum_{i\in\Z/p}a_{i}z^{-i-1}\in (\End M)[[z^{1/p},z^{-1/p}]],\quad a\in V$$
satisfying a number of conditions (\cite{Lep85,DLM00}).
We often denote it by $M$.
Note that an (untwisted) $V$-module is a $1$-twisted $V$-module
and that a $g$-twisted $V$-module is an (untwisted) $V^g$-module.
The {\it weight} of a homogeneous vector $v\in M_k$ is $k$, where $L(0)=\omega_{1}$.
Note that $L(0)v=kv$ if $v\in M_k$.
For a ($g$-twisted) irreducible $V$-module $M$, we use $\wt(M)$ to denote its lowest weight.

A VOA is said to be  {\it rational} if any module is completely reducible.
A rational VOA is said to be {\it holomorphic} if it itself is the only irreducible module up
to isomorphism.
A VOA is said to be {\it of CFT-type} if $V_0=\C\1$
, and is said to be {\it $C_2$-cofinite} if the codimension in $V$ of the subspace spanned by the vectors of form $u_{-2}v$, $u,v\in V$, is finite.
A module is said to be {\it self-dual} if its contragredient module (cf.\ \cite[Section 5.2]{FHL}) is isomorphic to itself.

\subsection{Simple current extension}

In this subsection, we review the theory of simple current extensions and lifts of automorphisms from \cite{SY03,Sh04}.

Let $V(0)$ be a simple, rational, $C_2$-cofinite VOA of
CFT-type and let $\{ V(\alpha) \mid \alpha \in A\}$ be a set of inequivalent irreducible $V(0)$-modules indexed by an abelian group
$A$. A simple VOA $V=\oplus_{\alpha\in A} V(\alpha)$ is called an {\it $A$-graded extension} of $V(0)$ if $V(0)$ is a full subVOA of
$V$ and $V$ carries a $A$-grading, i.e., $Y(x^\alpha,z)x^\beta\in V({\alpha+\beta})$ for any $x^\alpha \in V(\alpha)$, $x^\beta \in
V(\beta)$. In this case, the group $A^*$ of all irreducible characters of $A$ acts naturally on $V$: for $\chi\in A^*$,  $\chi\cdot v =\chi(\alpha) v$, $v\in V(\alpha)$. In other words, $V(\alpha)$ is an eigenspace of $A^*$ for all $\alpha\in A$ and $V(0)$ is the
fixed-points of $A^*$.
An irreducible $V(0)$-module $M$ is called a {\it simple current} if for any irreducible $V(0)$-module $U$, the fusion product $U\boxtimes M$ is irreducible.

If all $V(\alpha)$, $\alpha\in A$, are simple currents, then $V$ is referred to as an {\it $A$-graded simple current extension} of $V(0)$. 
Note that the fusion product $V(\alpha)\boxtimes V(\beta)=V(\alpha+\beta)$ holds for all $\alpha,\beta\in A$.
Let $S_A$ be the set of the isomorphism classes of the irreducible $V(0)$-modules $V(\alpha)$, $\alpha\in A$.
For an automorphism $g$ of $V(0)$, we set $S_A\circ g=\{M\circ g|\ M\in S_A\}$, where $M\circ g$ denotes the $g$-conjugate of $M$, i.e., $M\circ g=M$ as a vector space and the vertex operator $Y_{M\circ g}( u,z)= Y_M(gu,z)$, for $u\in V$.
Note that the lowest weights of $M$ and $M\circ g$ are the same.
Define  
$H_A=\{h\in\Aut V(0) \mid \ S_A\circ h=S_A\}.$    
Then we obtain the restriction homomorphism from $N_{\Aut V}(A^*)$ to $H_A$.
\begin{thm}[\cite{Sh04}; see also \cite{SY03}]\label{NCthm}
Let $V=\oplus_{\alpha\in A}V(\alpha)$ be an $A$-graded simple current extension of $V(0)$.
Then the restriction homomorphism is surjective and its kernel is $A^*$. That is, we have a short exact sequence
\begin{align*}
 0 \longrightarrow A ^\ast \longrightarrow N_{ \Aut V}( A ^\ast) \longrightarrow H_A \longrightarrow 0.
\end{align*}
\end{thm}

\subsection{Automorphism group of $V_{K_{12}}^\tau$}\label{Sec:AutV}
In this subsection, we will determine the automorphism group of $V_{K_{12}}^\tau$.

Let $K_{12}$ denote the Coxeter-Todd lattice of rank $12$.
Let $\langle\cdot|\cdot\rangle$ be the inner product of $\R\otimes_\Z K_{12}$.
For the construction and properties of $K_{12}$, see \cite{CS83}.
Let $K^*_{12}$ denote the dual lattice of $K_{12}$.
It is well-known that $K_{12}^*/K_{12}\cong\F_3^{6}$ and the map $q_{K_{12}}:K_{12}^*/K_{12}\to\F_3$, $q_{K_{12}}(x+K_{12})=3|x|^2\pmod3$ is a non-singular quadratic form of minus type on $K_{12}^*/K_{12}$.
Let $O(K_{12})$ denote the automorphism group of $K_{12}$ and let $\tau$ be a generator of $O_3(O(K_{12}))$, where $O_3(G)$ means the maximal normal $3$-subgroup of a finite group $G$. 
Then the action of $\tau$ on $K_{12}$ is fixed-point free and that $O(K_{12})/\langle\tau\rangle$ is isomorphic to the orthogonal group $O(K_{12}^*/K_{12},q_{K_{12}})$ of the shape $2.\Omega_6^-(3).2^2$.
Hence the automorphism group $O(K_{12})$ of $K_{12}$ has the shape $6.\Omega_6^-(3).2^2$.

Let $V_{K_{12}}$ be the lattice VOA associated to $K_{12}$.
There exists a lift of $\tau$ of order $3$ in $\Aut V_{K_{12}}$, which we denote by the same symbol $\tau$.

Set $W=V_{K_{12}}^\tau$.
We first review the properties of irreducible $W$-modules from \cite{CL14}.
It follows from \cite{TY13} and \cite{HKL} that $W$ is rational and $C_2$-cofinite.
It is known that there exist exactly $3^8$ non-isomorphic irreducible $W$-modules
$$S^a[x]:=V_{a+K_{12}}[x],\quad T^a[x;i]:=V_{K_{12}}^{T,a}(\tau^i)[x],$$
where $a\in K_{12}^*/K_{12}$, $i\in\{1,2\}$ and $x\in\{0,1,2\}$.
We often view $i$ and $x$ as elements in $\F_3$ by modulo $3$.
For the notation of irreducible $W$-modules, see \cite{CL14}.
Note that $S^0[0]=V_{K_{12}}^\tau$.
The lowest weights of irreducible $W$-modules satisfy the following:
\begin{equation}
 \wt\left(S^a[x]\right)\in\frac{|a|^2}{2}+\Z,\quad \wt\left({T}^a[x;i]\right)\in\frac{2(1+x+|a|^2)}{3}+\Z.\label{Eq:wt}
 \end{equation}
The fusion products of irreducible $W$-modules are given as follows (\cite[Section 5]{CL14}):
\begin{align}
\begin{split}\label{FR:K12}
 S^a[x]  +  S^b[ y ] = S^{ a +b} [ x +y],\qquad&
 S^a[x]  +  T^b[ y;i ] = T^{ b -{i}a } [ y+{i}x;i],\\
 T^a[x;1]  +  {T}^b[ y;2 ] = S^{b -a   } [ x -y],\qquad&
 T^a[x;i]  +  T^b[ y;i ] = {T}^{- (a +b)  } [- ( x +y);2i],
\end{split}
\end{align}
where $+$ means the fusion product. Recall that the fusion products are commutative and associative. 

Let $R(W)$ denote the set of isomorphism classes of irreducible $W$-modules.
We often identify an irreducible $W$-modules with its isomorphism class without confusion.
One can easily check that $(R(W),+)$ has an elementary abelian $3$-group structure of order $3^8$.
We view $(R(W),+)$ as an $8$-dimensional vector space over $\F_3$.
Let $q$ be the map $R(W)\to\F_3$ defined by 
\begin{equation}
q(M)\equiv 3\wt(M)\pmod3.\label{Def:q}
\end{equation}

\begin{prop}\label{thm:O83minus} The map $q$ is a non-singular quadratic form on $R(W)$ of minus type.
\end{prop}
\begin{proof} Let $B(M_1,M_2):=(q(M_1+M_2)-q(M_1)-q(M_2))/2\in\F_3$.
By \eqref{Eq:wt}, we obtain
\begin{align}
\begin{split}\label{Eq:B}
&B(S^a[x],S^b[y])=3\langle a|b\rangle,\qquad B(S^a[x],T^b[y;i])=-{i}(3\langle a|b\rangle+x),\\ &B(T^a[x;i],T^b[y;j])={i}{j}(3\langle a|b\rangle-x-y+1).
\end{split}
\end{align}
Hence by \eqref{FR:K12}, we can easily see that $B(\,,\,)$ is bilinear on $R(W)$.
Clearly $B(\,,\,)$ is non-singular.
By \eqref{Eq:B}, we have the orthogonal direct sum of quadratic spaces as follows:$$R(W)=\{S^a[0]\mid a\in K_{12}^*/K_{12}\}\perp\Span_{\F_3}\{ S^0[1],T^0[2;1]\}.$$
The former is isomorphic to the $6$-dimensional quadratic space $(K_{12}^*/K_{12},q_{K_{12}})$ of minus type and the latter is a $2$-dimensional quadratic space of plus type.
Hence the type of the quadratic space $(R(W),q)$ is minus.
\end{proof}

The following theorem can be proved by the exactly the same argument as in \cite{LY13}.

\begin{prop}[cf.\ {\cite[Theorem 5.15]{LY13}}]\label{thm:ses:C_AutVL: C_OL}
 Let $L$ be a positive-definite even rootless lattice.
Let $\nu$ be a fixed-point free isometry of $L$ of prime order $p$.
Let $\nu$ also denote a lift of $\nu$ in $\Aut V_L$ of order $p$.
Then we have an exact sequence
\begin{equation}\label{eq:ses:C_AutVL: C_OL}
   1 \longrightarrow \Hom(L/(1-\nu)L, \mathds{Z}_p)
  \longrightarrow  N_{\Aut V_L}({\nu})
   \longrightarrow  N_{O(L)}(\nu) \longrightarrow 1.
\end{equation}
\end{prop}

Let $g\in\Aut W$.
Consider the map $\rho(g)$: $M\mapsto M\circ g$ on $R(W)$.
Since this action preserves the fusion products, $\rho(g)$ is a group isomorphism of $R(W)$.
In addition, it preserves the lowest weights of $W$-modules and hence $\rho(g)$ preserves the quadratic form $q$, also.
Hence we obtain a group homomorphism $\rho:\Aut W\to O(R(W),q)$, where $O(R(W),q)$ is the orthogonal group.

\begin{prop}\label{Prop:AutVK12tau} The map $\rho$ is a group monomorphism and the image is the index $2$ subgroup of $O(R(W),q)$ of shape $\Omega_8^-(3){:}2$.
In particular, $\Aut V_{K_{12}}^\tau$ has the shape $\Omega_8^-(3){:}2$.
\end{prop}
\begin{proof} Recall that $V_{K_{12}}$ is a simple current extension of $W$ graded by an cyclic group $A$ of order $3$.
Set $N=N_{\Aut V_{K_{12}}}(A^*)/A^*$ and $S_A=\{S^0[0],S^0[1],S^0[2]\}$.
By Theorem \ref{NCthm} and $V_{K_{12}}\cong \bigoplus_{x=0}^2 S^0[x]$ as $W$-modules, we obtain
\begin{equation}
N=\{g\in \Aut W\mid S_A\circ g=S_A \}.\label{Eq:N}
\end{equation}
Clearly, $\Ker\varphi$ is a subgroup of $N$.
By Proposition \ref{thm:ses:C_AutVL: C_OL} and the fact that $O_3(O(K_{12}))=\langle\tau\rangle$, we obtain an exact sequence
\begin{equation}
1\to\Hom( K_{12}/ ( 1 - \tau) K_{12}, \mathbb{F}_3)\to N\to O(K_{12})/\langle\tau\rangle\to 1.\label{Exact:N}
\end{equation}
Since $O(K_{12})/\langle\tau\rangle$ acts faithfully on $K_{12}^*/K_{12}$, one can see that $N$ also acts faithfully on $R(W)$ by a similar calculation as in \cite[Proposition 2.9]{Sh04}.
Hence $\Ker\varphi$ is trivial.

Next we determine $\Aut W$.
In \cite{LY13}, a subgroup $H$ of $\Aut W$ of shape $\Omega_8^-(3){:}2$ was constructed.
Since $|O(R(W),q)/(\Omega_8^-(3){:}2)|=2$, we have $\Aut W=H$ or $\Aut W\cong O(R(W),q)$.
Suppose that $\Aut W\cong O(R(W),q)$.
Then $\Aut W$ has the central element $h$ of order $2$, and for $M\in R(W)$, $M\circ h=M'$, where $M'\in R(W)$ such that $M\boxtimes M'=W$.
Hence $h$ preserves $\{S^0[0],S^0[1],S^0[2]\}$, and by \eqref{Eq:N}, we have $h\in N$.
Since $Z(O(K_{12})/\langle\tau\rangle)\cong\Z_2$ and its action on $\Hom( K_{12}/ ( 1 - \tau) K_{12}, \mathbb{F}_3)$ is non-trivial in \eqref{Exact:N}, we have $Z(N)=1$, which is a contradiction.
Thus we obtain $\Aut W=H$.
\end{proof}

\begin{rmk} In fact, the automorphism group of $V_{K_{12}}^\tau$ of shape $\Omega^-_8(3){:}2$ is isomorphic to ${}^+\Omega^-(8,3)$ (see \cite[Theorem 5.64]{LY13}).
\end{rmk}

\section{Automorphism group of the holomorphic VOA $V^\sharp$}
In this section, we review a construction of the holomorphic VOA $V^\sharp$ from \cite{Mi13a} and construct some $3$-local subgroups of the automorphism group $\Aut V^\sharp$ of $V^\sharp$.
By using those subgroups, we will show that $\Aut V^\sharp$ is finite.
In addition, we will show that $\Aut V^\sharp$ is isomorphic to the Monster simple group based on a $3$-local characterization of the Monster simple group.

\subsection{Holomorphic VOA $V^\sharp$}
In this subsection, we review the construction of the holomorphic VOA $V^\sharp$ and its properties from \cite{Mi13a}.

Let $ \Lambda$ denote the Leech lattice and let $ \tau$ be a fixed-point free isometry of $ \Lambda$ of order $3$. 
There exists a lift of $\tau$ of order $3$ in $\Aut V_\Lambda$, which we denote by the same symbol $\tau$.
The orbifold VOA $V_\Lambda^\tau$ is rational and $C_2$-cofinite \cite{TY13}.
It is well-known \cite{DLM00} that the Leech lattice VOA $ V_ \Lambda$ has exactly one irreducible $ \tau^i$-twisted module for $ i = 1,2$. We denote these twisted modules by $V_\Lambda ^{T_1}(\tau)$ and $ V_ \Lambda^{T_2}(\tau^2)$, respectively.
We also define the $ V_\Lambda^\tau$-module
\begin{align}
 V^ \sharp :=&   V_ \Lambda^\tau \oplus  (V_ \Lambda ^{T_1}(\tau))_ \mathbb{Z} \oplus  (V_ \Lambda ^{T_2}(\tau^2))_ \mathbb{Z},\label{Eq:Vsharp}
\end{align}
where $ (V_\Lambda ^{T_i}(\tau^i))_ \mathbb{Z} $ is the subspace of $ V_\Lambda ^{T_i}(\tau^i)$ of integral weights.

\begin{thm}{\bf (\cite{Mi13a})}\label{Thm:sharp} The $V_\Lambda^\tau$-module $V^\sharp$ has a VOA structure as a $\Z_3$-graded simple current extension of $V_ \Lambda^\tau$.
In addition, $V^\sharp$ is of CFT-type, $C_2$-cofinite, holomorphic and $V^\sharp_1=0$.
\end{thm}

\subsection{$3$-local characterization of the Monster simple group}
In this subsection, we recall a $3$-local characterization of the Monster simple group from \cite{SS10}. 

\begin{thm}[{\cite[Theorem  1.1]{SS10}}]\label{thm:3local_monster}
   Let $G$ be a finite group and $ S$ a Sylow $3$-subgroup of $G$.
Let $ H_1$ and $H_2$ be subgroups of $G$ containing $S$ such that
\begin{enumerate}[{\rm (i)}]
   \item \label{3Monster:1} $H_1 = N_G( Z( O_3(H_1)))$, $ O_3( H_1)$ is an  extraspecial group of order $ 3^{13}$,\\ $H_1/{ O_3 (H_1)} \cong 2 \Suz{:} 2$ and $C_{ H_1} ( O_3 ( H_1)) = Z( O_3(H_1))$;
   \item\label{3Monster:2} $ H_2/ O_3( H_2) \cong \Omega_8^-(3 ).2$, $ O_3( H_2)$ is an elementary abelian group of order $ 3^8$ and $ O_3 (H_2)$ is a natural $ H_2/ O_3( H_2)$-module;
   \item\label{3Monster:3}  $ ( H_1 \cap H_2) / O_3 ( H_2)$ is an extension of an elementary abelian group of order $ 3^6$ by $ 2 \cdot PSU_4(3){:}2^2$.
  \end{enumerate}
Then $G$ is isomorphic to the largest sporadic simple group, the Monster.
\end{thm}

We will construct explicitly certain subgroups $H_1$ and $H_2$ of $\Aut V^\sharp$ such that $H_1$ and $H_2$ satisfy the hypotheses (i), (ii) and (iii) of the theorem above.

\subsection{A subgroup of the shape $3^{1+12}( 2\Suz{:}2)$ of $\Aut V^\sharp$}

In this subsection, we construct a subgroup of $ \Aut V^ \sharp$ satisfying (i) of Theorem \ref{thm:3local_monster}.

Since $V^\sharp$ is a $\Z_3$-graded simple current extension of $V_\Lambda^\tau$ as in \eqref{Eq:Vsharp}, there is a natural automorphism $\tau'$ of order $3$ which acts on $V^\sharp$ as $1$ on $V_ \Lambda^\tau $ , as $\xi$ on $  (V_ \Lambda ^{T_1}(\tau))_\mathbb{Z}$, and as $\xi^2$ on $  (V_ \Lambda ^{T_2}(\tau^2))_\mathbb{Z} $, where $\xi=\exp(2\pi\sqrt{-1}/3)$.

\begin{prop}\label{prop:2.6} Set $H_1=N_{\Aut V^\sharp} (\langle \tau' \rangle  )$.
Then $O_3(H_1)$ is an extra-special $3$-group of order $3^{13}$, $Z(O_3(H_1))=\langle\tau'\rangle$, $H_1/O_3(H)\cong 2\Suz{:}2$, and $C_{ H_1} ( O_3 ( H_1)) = Z( O_3(H_1))$.
In particular, the subgroup $H_1=N_{\Aut V^\sharp} (\langle \tau' \rangle  )$ satisfies the hypothesis \eqref{3Monster:1} of Theorem \ref{thm:3local_monster}.
\end{prop}
 \begin{proof}
 Let $A$ be the cyclic group of order $3$ corresponding the grading of $V^\sharp$ as a simple current extension of $V_\Lambda^\tau$.
 Note that $A^*=\langle\tau'\rangle$.
Let $ S_A$ be the set of isomorphism classes of three irreducible $V_\Lambda^\tau$-modules $  V_\Lambda^\tau$,  $(V_ \Lambda ^{T_1}(\tau))_\mathbb{Z}$, and $(V_ \Lambda ^{T_2}(\tau^2))_\mathbb{Z}$.
Set $H_A=\{h\in\Aut  V_\Lambda^\tau \mid \ S_A\circ h=S_A\}.$
Recall that $V_\Lambda^\tau$ has exactly nine irreducible modules and only $(V_ \Lambda ^{T_1}(\tau))_ \mathbb{Z}$ and $(V_ \Lambda ^{T_2}(\tau^2))_\mathbb{Z} $  have the lowest weight $2$ (see \cite{Mi13a,TY13}). Hence $ S_A$ is invariant under the action of  $ \Aut V_\Lambda^\tau$.
Thus we obtain $H_A=\Aut V_\Lambda^\tau$.
By Theorem \ref{NCthm}, we obtain a short exact sequence
\begin{equation}
1\longrightarrow \langle\tau'\rangle\longrightarrow H_1
\longrightarrow \Aut V_\Lambda^\tau\longrightarrow1.\label{Eq:H1}
\end{equation}

Let us determine $\Aut V_\Lambda^\tau$.
For $i=1,2$, let $V_\Lambda^\tau[i]$ be the eigenspace of $\tau$ in $V_\Lambda$ with eigenvalue $\xi^i$.
By \cite{Mi13a}, 
\begin{align*}
V_ \Lambda = V_\Lambda^\tau \oplus V_ \Lambda^\tau[1]  \oplus V_ \Lambda^\tau[2]
\end{align*}
is a $\Z_3$-graded simple current extension of $V_\Lambda^\tau$.
Let $B$ be the cyclic group of order $3$ corresponding to the $\Z_3$-grading.
Note that $B^*=\langle\tau\rangle$.
Let $ S_B$ be the set of isomorphism classes of irreducible $V_\Lambda^\tau$-modules $   V_\Lambda^\tau$, $V_ \Lambda^\tau[1]$, and $V_ \Lambda^\tau[2]$.
Set $H_B := \{ h \in \Aut  V_ \Lambda^ \tau \mid S_B \circ h = S_B \}$.
Since $ V_ \Lambda^\tau[1]  $ and $ V_ \Lambda^\tau[2]$ are the only irreducible modules of $ V_ \Lambda^\tau$ with the  lowest weight $1$,
we have $H_B = \Aut V_ \Lambda ^ \tau$.
Then by Proposition \ref{NCthm}, we have a short exact sequence
\begin{align*}
 0 \longrightarrow \langle {\tau}\rangle  \longrightarrow N_{ \Aut V_ \Lambda } ( \langle {\tau}\rangle)  \longrightarrow  \Aut V_ \Lambda ^ \tau \longrightarrow 0.
\end{align*}
Hence by Proposition \ref{thm:ses:C_AutVL: C_OL}, we obtain
\begin{equation}
\Aut V_ \Lambda ^ \tau \cong N_{\Aut V_\Lambda}(\langle {\tau} \rangle)/\langle {\tau}\rangle\cong 3^{12}.N_{O(\Lambda)}(\langle \tau\rangle)/\langle \tau\rangle.\label{Eq:aut1}
\end{equation}
According to \cite[p180]{Atlas}, we have $N_{O(\Lambda)}(\langle\tau\rangle)/\langle \tau\rangle\cong 2\Suz{:}2$.
By \eqref{Eq:H1} and \eqref{Eq:aut1}, $O_3(H_1)$ is a $3$-group of order $3^{13}$ and $H_1/O_3(H_1)\cong 2\Suz{:}2$.

Next, we will show that the subgroup $O_3(H_1)$ is extra-special.
Let $Z$ be the center of $O_3(H_1)$.
Note that $O_3(H_1)/\langle\tau'\rangle$ is an elementary abelian $3$-group of order $3^{12}$ and that $\tau'\in Z$.
Hence it suffices to show that $Z=\langle\tau'\rangle$.
Since the action of $H_1$ on $O_3(H_1)/\langle\tau'\rangle$ by the conjugation is irreducible and it preserves $Z$, we have $Z=O_3(H_1)$ or $Z=\langle\tau'\rangle$. 

Let us consider the action of $O_3(H_1)$ on the $V_\Lambda^\tau$-submodule $(V_{\Lambda}^{T_i}(\tau^i))_\Z = (S[\tau^i]\otimes T_i)_\Z$ of $V^\sharp$, where $S[\tau^i]$ is a $\tau^i$-twisted Heisenberg algebra and $T_i$ is a $3^{6}$-dimensional irreducible module over $\C$ for the $\tau^i$-twisted central extension of $\Lambda$ by $\Z/6\Z$ associated to the commutator map in \cite[Remark 2.2]{Dl96}.
Clearly, $O_3(H_1)$ preserves $(V_{\Lambda}^{T_i}(\tau^i))_\Z$.
Let $g\in O_3(H_1)$.
Then $g=1$ on the subspace $M(1)^\tau$ of $V_\Lambda^\tau$, where $M(1)$ is the Heisenberg subVOA of $V_\Lambda$.
Hence $g$ acts on $T_i$ by $g(x\otimes t)=x\otimes g(t)$, where $x\in S[\tau]$ and $t\in T_i$.

Suppose $Z=O_3(H_1)$. Then $O_3(H_1)$ is an elementary abelian $3$-group of order $3^{13}$ by the irreducibility.
It follows from $\dim T_i=3^{6}$ that there exists $1\neq h\in O_3(H_1)$ such that $h=1$ on $T_i$.
Then $h=1$ on $(S[\tau]\otimes T_i)_\Z$.
Since $V_\Lambda^\tau$ is simple, $(S[\tau]\otimes T_i)_\Z$ is a faithful $V_\Lambda^\tau$-module.
Thus $h=1$ on $V_\Lambda^\tau$.
In addition, $V^\sharp$ is a $\Z_3$-graded simple current extension of $V_\Lambda^\tau$, we have $h=1$ on $V^\sharp$, which contradicts the choice of $h$.
Therefore $Z=\langle\tau'\rangle$, and $O_3(H_1)$ is an extra-special $3$-group $3^{1+12}$.

Finally, let us show $C_{H_1}(O_3(H_1))=Z(O_3(H_1))$.
The conjugate action of $H_1/O_3(H_1)$ on $O_3(H_1)/Z(H_1)$ is faithful.
Hence $C_{H_1}(O_3(H_1))\subset Z(O_3(H_1))=\langle\tau'\rangle$.
Clearly, $Z(O_3(H_1))\subset C_{H_1}(O_3(H_1))$.
Thus we have $C_{H_1}(O_3(H_1))=Z(O_3(H_1))$.
\end{proof}

\subsection{A subgroup of the shape $3^8 .\Omega^-_8(3).2$}
In this subsection, we construct a subgroup of $\Aut V^\sharp$ satisfying the hypothesis (ii) of Theorem \ref{thm:3local_monster}.


Let us consider an explicit embedding of $K_{12}\perp K_{12}$ into the Leech lattice.  Recall from \cite[(A.3) and (A.4)]{KLY03} that the Leech lattice $\Lambda$ can be constructed as an overlattice of a sublattice $N$ of the type  $\sqrt{2}A_2^{12}$ and a pair of the $\Z_3$-code $\mathcal{D}$ and the $\F_4$-code $\mathcal{C}=\mathcal{H}\oplus \mathcal{H}$ as glue codes, where $\mathcal{H}$ is the hexacode and $\mathcal{D}$ is the ternary code given by the generator matrix 

\[
\left(
\begin{array}{cccccccccccc}
1&1 &0&0 &0&0 &1&0 &0&0 &0&0 \cr
1&2 &0&0 &0&0 &0&1 &0&0 &0&0 \cr
0&0 &1&1 &0&0 &0&0 &1&0 &0&0 \cr
0&0 &1&2 &0&0 &0&0 &0&1 &0&0 \cr
0&0 &0&0 &1&1 &0&0 &0&0 &1&0  \cr
0&0 &0&0 &1&2 &0&0 &0&0 &0&1
\end{array}
\right).
\]
More precisely, $N^*/N\cong\F_4^{12}\times \Z_3^{12}$ and $\Lambda/N\cong \mathcal{C}\times\mathcal{D}$, where $N^*$ is the dual lattice of $N$.
Notice that $\mathcal{D}$ is isomorphic to a sum of $3$ ternary tetra-codes and 
that the overlattice of $N$ corresponding to $\mathcal{C}$ is isomorphic to $K_{12}\perp K_{12}$. 

\begin{lem}\label{perK12}
There exists an isometry in $C_{O(\Lambda)}(\tau)$ such that 
it stabilizes $L_{\mathcal{C}}$ and permutes the two orthogonal copies of $K_{12}$. 
\end{lem}

\begin{proof}
Let $s= (1, 7)(2, 8)(3, 9) (4, 10)(5,11) (6, 12)$ be the permutation of the $12$ coordinates of $N^*/N$ with respect to the generator matrix above. 
Let $\varepsilon$ be the automorphism of $N^*/N$ of order $2$ which acts as $1$ on $\{1,2,3,4,5,6\}$ and as $-1$ on $ \{7,8,9,10,11,12\}$. 
Then $h'=\varepsilon s$ stabilizes both codes $\mathcal{D}$ and $\mathcal{C}$ and induces a transposition on the two copies of $\mathcal{H}$ in $\mathcal{C}$. 
Hence, it induces an isometry $h\in O(\Lambda)$ such that $h(L_{\mathcal{C}})=L_{\mathcal{C}}$ and it acts on  the two orthogonal copies of $K_{12}$ as a transposition. 

By viewing $\tau$ as the isometry of $\Lambda$ of order $3$ induced from that of the sublattice $N\cong \sqrt2A_2^{12}$, we know that $K_{12}\perp K_{12}$ is invariant under $\tau$.
Notice that the fixed-point free isometry of $\sqrt{2}A_2$ of order $3$ fixes the three cosets
in $2(\sqrt{2}A_2)^*/ \sqrt{2}A_2$ and acts as a cyclic permutation on the three non-zero elements of $\frac{1}2 \sqrt{2}A_2/ \sqrt{2}A_2 (\cong \mathbb{F}_4)$. Now it is easy to check that $h$ commutes with $\tau$.  
\end{proof}

By the embedding $K_{12}\perp K_{12}\subset\Lambda$, we have $(V_{K_{12}})^{\otimes2}\subset V_{\Lambda}$.
Then we have an embedding $(V_{K_{12}}^\tau)^{\otimes 2}\subset V_\Lambda^\tau$.
Set $W=V_{K_{12}}^\tau$ and $W^2=(V_{K_{12}}^\tau)^{\otimes 2}$.

\begin{lem}\label{Lem:2K12} There exists an automorphism $\mu $ of $V^\sharp$ such that $\mu (W^2)=W^2$, $\mu (W\otimes \1)=\1\otimes W$ and $\mu (\1\otimes W)=W\otimes \1$.
\end{lem}
\begin{proof} It follows from Lemma \ref{perK12} and Propositions \ref{thm:ses:C_AutVL: C_OL} and \ref{prop:2.6}. 
\end{proof}

Under the fusion product, $R=R(W)$ forms an $8$-dimensional vector space over $\F_3$.
The (non-singular) quadratic form $q$ on $R$ of minus type was given in Section \ref{Sec:AutV}.
Since any irreducible $W^{2}$-module is a simple current, $V^\sharp$ is a simple current extension of $W^{2}$.
Let $S^\sharp$ be the subgroup of $R^2(=R\times R)$ such that $V^\sharp\cong\bigoplus_{(M_1,M_2)\in S^\sharp} M_1\otimes M_2$ as $W^{2}$-modules.
Note that the map $q^2: R^2\to\F_3$, $q^2(a,b)=q(a)+q(b)$ is a quadratic form on $R^2$.

First, we determine the automorphism group of $W^2$.

\begin{lem}\label{Lem:autV} The automorphism group of $W^2$ is $\Aut W\wr\Z_2$.
\end{lem}

\begin{proof} We view $\Aut W^2$ as a subgroup of $O(R^2,q^2)$.
Then $\Aut W^2$ preserves not only $q^2$ but also the lowest weights of elements of $R^2$.
Hence $\Aut W^2$ stabilizes the subset of $R^2$ with the lowest weight $2/3$, $\{(M,W),(W,M)\in R^2\mid q(M)=1\}$.
By a similar argument as in \cite[Proposition 4.16]{Sh11}, $\Aut W^2$ is a subgroup of $O(R,q)\wr \Z_2$.
Clearly, $\Aut W\wr\Z_2\subset \Aut W^2$.
By Proposition \ref{Prop:AutVK12tau}, $Z(O(R,q))\cap \Aut W=1$.
Since $O(R,q)\cong Z(O(R,q))\times \Aut W$, we have this lemma.
\end{proof}

Next, we show some properties of $S^\sharp$.

\begin{lem}\label{Lem:Omega83}
\begin{enumerate}[{\rm (1)}]
\item The subspace $S^\sharp$ of $(R^2,q^2)$ is maximal totally singular, and $|S^\sharp|=3^8$.
\item There exists a unique linear isomorphism $\eta:R\to R$ such that $S^\sharp=\{(M,\eta(M))\mid M\in R\}$.
\item The map $\eta$ in (2) satisfies $q(\eta(M))=-q(M)$ for all $M\in R$.
In particular, $\eta$ is an outer automorphism of $O(R,q)$.
\end{enumerate}
\end{lem}
\begin{proof} 
Let $(M_1,M_2)\in S^\sharp$.
By \eqref{Def:q}, we have 
$$q^2(M_1,M_2)=q(M_1)+q(M_2)\equiv 3\wt(M_1\otimes M_2)\mod 3.$$
Since $M_1\otimes M_2$ appears in $V^\sharp$ as an irreducible $W^2$-submodule, $\wt(M_1\otimes M_2)\in\Z$, and $q^2(M_1,M_2)=0$ in $\F_3$.
Hence $S^\sharp$ is totally singular.
It follows from $|\Lambda/(K_{12}\perp K_{12})|=3^6$ that $V_\Lambda^\tau$ is a direct sum of (non-isomorphic) $3^6$ irreducible $V_{K_{12}\perp K_{12}}^\tau$-modules.
Hence $V_\Lambda^\tau$ is a direct sum of (non-isomorphic) $3^7$ irreducible $W^2$-modules.
Since $V^\sharp$ is a $\Z_3$-graded simple current extension of $V_\Lambda^\tau$, we have $|S^\sharp|=3^8$.
Since $(R^2,q^2)$ is a $16$-dimensional quadratic space of plus type, 
$S^\sharp$ is maximal, and we obtain (1).

Let $(M,W)$ or $(W,M)\in S^\sharp$.
It follows from $\wt (W)=0$ that $\wt(M)=\wt(M\otimes W)$.
Since $(S^\sharp)_1=0$ and $\wt(M)\in\{0, 2/3,4/3,1\}$, we have $\wt(M)=0$.
Hence $M=W$, and there exists a subset $Q$ of $R$ and an injective map $\eta: Q\to R$ such that $S^\sharp=\{(M,\eta(M))\mid M\in Q\}$.
Since $S^\sharp$ is a subspace, so is $Q$.
By (1), $\dim R=\dim S^\sharp$.
Hence $Q=R$ and $\eta$ is a linear isomorphism, which proves (2).

It follows from $\wt(M\otimes\eta(M))\in\Z$ that  $q(M)+q(\eta(M))=0$ for all $M\in R$, that is, $q(\eta(M))=-q(M)$ for all $M\in R$.
For any $g\in O(R,q)$ and $M\in R$, we have 
\[
q(\eta g\eta^{-1}(M))=-q(g\eta^{-1}(M)))=-q(\eta^{-1}(M))=q(M).
\]
Hence $\eta$ is an outer automorphism of $O(R,q)$, and we obtain (3).
\end{proof}

Let $\rho: \Aut W \to O(R,q)$ be the group monomorphism obtained in Proposition \ref{Prop:AutVK12tau}. Set  $P=\rho(\Aut W)$, which has the shape $\Omega^-_8(3){:}2$. Since $\eta$ is an outer automorphism (cf. Lemma \ref{Lem:Omega83} (3)), we have $\eta P\eta^{-1} \subset O(R,q)$ but $\eta P\eta^{-1} \not\subset P$. 

Let $K$ be the commutator subgroup of $\rho(\Aut W)$, which has the shape $\Omega^-_8(3)$. Then for $g\in K$, we have $\eta g\eta^{-1}\in K$.
Hence $K=\eta K\eta^{-1}= \eta P\eta^{-1}\cap  P $.
Let $$\gamma: K\to K\times K,\qquad g\mapsto (g,\eta g\eta^{-1}).$$
Clearly, $\gamma$ is injective, and by Lemma \ref{Lem:Omega83} (2), $\gamma(K)$ preserves $S^\sharp$.

\begin{prop}\label{StabS} 
Let $\bar{\mu}$ be the automorphism of $W^2$ defined as the restriction of the automorphism $\mu$ of $V^\sharp$ given in Lemma \ref{Lem:2K12}.  
The stabilizer of $S^\sharp$ in $\Aut W^2$ is generated by $\gamma(K)$ and $\bar\mu$, and it has the shape $\Omega^-_8(3).2$.
\end{prop}
\begin{proof} Let $G$ denote the stabilizer of $S^\sharp$ in $\Aut W^2$.
We have already mentioned that $\gamma(K)\subset G$.
Since $\mu$ is an automorphism of $V^\sharp$, we have $\bar\mu \in G$.
Recall from Lemma \ref{Lem:autV} that $\Aut W^2\cong\Aut W\wr\Z_2$.
Hence any automorphism of $W^2$ stabilizes both the subVOAs $W\otimes\1$ and $\1\otimes W$ or permutes them.

Let $f\in G$ such that $f$ stabilizes both $W\otimes \1$ and $\1 \otimes W$. 
Then $f= (f_1,f_2)\in (\Aut W\times \Aut W)\cap G$. In this case, $f_2\eta(M)=\eta (f_1(M))$ for all $M\in R$.
Hence $f_2=\eta f_1 \eta^{-1}$ on $R$, and $(f_1,f_2)\in\gamma(K)$.

Now let $f \in G$ such that $f$ permutes $W\otimes \1$ and $\1 \otimes W$. 
Then $f\bar{\mu}^{-1}\in G$ stabilizes both $W\otimes \1$ and $\1 \otimes W$. Hence, by the argument above, we have $f\bar{\mu}^{-1}\in \gamma(K)$, and hence $f\in\gamma(K)\bar{\mu}$.
Since $\bar{\mu}^2$ stabilizes both $W\otimes \1$ and $\1 \otimes W$, we have $\bar{\mu}^2\in\gamma(K)$, which implies that the shape of $\langle\gamma(K),\bar{\mu}\rangle$ is $\Omega^-_8(3).2$.
\end{proof}


By Lemma \ref{Lem:Omega83} (1), $V^\sharp$ is a simple current extension of $W^2$ graded by $S^\sharp\cong \Z_3^8$.
By Lemma \ref{Lem:Omega83} (2), the map from $S^\sharp$ to $R$ given by $(M,\eta(M))\mapsto M$ is a linear isomorphism.
We identify $S^\sharp$ with $R$ by this map and we view $V^\sharp$ is an $R$-graded simple current extension of $W^2$.
Let $\varphi_M$ denote the linear automorphism of $V^\sharp$ associated to $M\in R$ such that 
\begin{equation}
\varphi_M(v)=\xi^{B(M,U)}v\qquad \text{for}\quad v\in U\otimes\eta(U),\quad U\in R,\label{Eq:phiM}
\end{equation}
where $B(\cdot,\cdot)$ is the bilinear form on $R$ given in \eqref{Eq:B} and $\xi=\exp(2\pi\sqrt{-1}/3)$.
By the $R$-grading of $V^\sharp$, 
we have $\Aut V^\sharp\supset R^*=\{\varphi_M\mid M\in R\}\cong\Z_3^8$.
By Theorem \ref{NCthm} and Proposition \ref{StabS}, we obtain the following:

\begin{prop}\label{NA} The normalizer of $R^*$ in $\Aut V^\sharp$ has the shape $3^8.\Omega^-_8(3).2$.
In particular, the subgroup $H_2=N_{\Aut V^\sharp} (R^*)$ satisfies the hypothesis \eqref{3Monster:2} of Theorem \ref{thm:3local_monster}.
\end{prop}

The following lemma will be used in the later sections.

\begin{lem}\label{Lem:finite} There exists an elementary abelian $3$-subgroup $E$ of $R^*$ of order $3^2$ such that $C_{\Aut V^\sharp}(a)$ is finite for all $a\in E\setminus\{1\}$.
\end{lem}
\begin{proof} Let $x\in\{1,2\}$.
Recall that $S^0[x]$ is the eigenspace of $\tau$ in $V_{K_{12}}$ with eigenvalue $\xi^x$.
It follows from $\wt(S^0[x])\in\Z$ that $S^0[x]$ is non-zero singular in $(R,q)$.
By \eqref{Eq:B}, $S^0[x]^\perp=\{S^a[y]\mid a\in K_{12}^*/K_{12},y\in\F_3\}$, where $S^0[x]^\perp$ is the orthogonal complement of $S^0[x]$ with respect to $B(\cdot,\cdot)$.
Hence  $V_{\Lambda}^\tau\cong\bigoplus_{M\in S^0[x]^\perp}M\otimes\eta(M)$ as $W^2$-modules.
By \eqref{Eq:B} again, $\varphi_{S^0[x]}=(\tau')^{x}$ on $V^\sharp$, and hence $C_{\Aut V^\sharp}(\varphi_{S^0[x]})$ is finite by Proposition \ref{prop:2.6}.
Note that if $a\in\Aut V^\sharp$ is conjugate to $\varphi_{S^0[x]}$, then $C_{\Aut V^\sharp}(a)$ is also finite.
Since $N_{\Aut V^\sharp}(R^*)/R^*$ contains the subgroup of shape $\Omega^-_8(3)$ that acts on $R^*$ as a natural module, $\varphi_M$ associated to a non-zero singular element $M\in R$ is conjugate to $\varphi_{S^0[1]}$ or $\varphi_{S^0[2]}$ under $N_{\Aut V^\sharp}(R^*)$.
Hence for a $2$-dimensional totally singular subspace $X$ of $R$, the abelian group $\{\varphi_M\mid M\in X\}$ satisfies the desired condition.
\end{proof}

\subsection{The intersection of the two subgroups}

In Sections 3.3 and 3.4, we have constructed the following two subgroups of $V^\sharp$:
\[
\begin{split}
&H_1= N_{\Aut V^\sharp }(\langle \tau'\rangle) \cong 3^{1+12}.(2.\Suz{:}2), \\
&H_2= N_{\Aut V^\sharp }(R^*) \cong 3^8. \Omega^-_8(3).2.
\end{split}
\]
In this subsection, we check the remaining condition about $H_1\cap H_2$ in Theorem \ref{thm:3local_monster}.

\begin{lem}\label{Lem:inter}
The intersection $H_1 \cap H_2$ has the shape $3^8 . ( 3^6 . (2  \mathrm{PSU}_4(3) . 2^2)).$
In particular, the subgroup $H_1\cap H_2$ satisfies the hypothesis \eqref{3Monster:3} of Theorem \ref{thm:3local_monster}.
\end{lem}

\begin{proof} Set $W=V_{K_{12}}^\tau$ and $R=R(W)$ as in Section 3.4.
We consider the inclusion $W^2\subset V_\Lambda^\tau \subset V^\sharp$ based on the inclusion $K_{12}\perp K_{12}\subset \Lambda$ in Section 3.4.
Then $\tau'\in O_3(H_2)$, and $O_3(H_2)$ is a normal subgroup of $H_1\cap H_2$.
In addition, $$(H_1\cap H_2)/O_3(H_2)=\{g\in H_2/O_3(H_2)\mid g\langle\tau'\rangle g^{-1}=\langle\tau'\rangle\}.$$
Recall from the proof of Lemma \ref{Lem:finite} that $\tau'=\varphi_{S^0[1]}$.
Hence $(H_1\cap H_2)/O_3(H_2)$ is the stabilizer of the $1$-dimensional singular subspace $\Span_{\F_3}\{S^0[1]\}$ in $\Omega^-_8(3).2$, and has the shape $3^6: (2.PSU_4(3).2^2)$ by 
\cite[p141]{Atlas}.
Therefore $H_1\cap H_2 \cong 3^8.(3^6: (2.PSU_4(3).2^2))$.
\end{proof}

\subsection{Finiteness of groups}
In this subsection, we show that $\Aut V^\sharp$ is finite. 

First, we prove the finiteness of the automorphism group of a VOA under a slightly general situation.
The following is a modification of \cite[Theorem (iv)]{T84} and \cite[Proposition 2.25]{Sh14}.

\begin{thm}\label{finiteness}
Let $G$ be an algebraic group and $p$ a prime number. 
Assume that $G$ contains an elementary abelian $p$-subgroup $E$ of order  $p^2$ such that $C_{G}(a)$ is finite for any $a\in E\setminus \{1\}$. 
Then $G$ is finite.
\end{thm}

\begin{proof}
Let $G^0$ be the connected component of $G$ containing the identity. Then $G^0$ is normal in $G$. For any $x\in G$, the map $\zeta_x: G^0 \to G^0, g\mapsto x^{-1} gx$ defines an automorphism in $\Aut G^0$.
Let $Lie(G^0)$ be the Lie algebra of $G^0$. Then the canonical group homomorphism from $\Aut G^0$ to
$\Aut Lie(G^0)$ is injective.

Now suppose $G^0\neq \{1\}$. 
Let $a\in E\setminus\{1\}$.
Since $C_G(a)$ is finite, $\zeta_a$ induces a non-trivial automorphism in $\Aut(Lie(G^0))$, and the fixed-point of $\zeta_a$ is $0$, only.
Let $b\in E\setminus\la a\ra$.  Since $\zeta_a$ commutes with $\zeta_b$, there is a common eigenvector $v$ in $Lie(G^0)$.
That means
\[
\zeta_a(v)=\exp({2\pi\sqrt{-1}i/p}) v\quad \text{ and }  \quad \zeta_b(v)=\exp({2\pi\sqrt{-1}j/p}) v
\]
for some $1 \leq i,j \leq p-1$.
Since $p$ is prime, we have $\zeta_a^j\zeta_b^{-i}=\zeta_{a^jb^{-i}}\in E\setminus \{1\}$, and 
$\zeta_a^{j}\zeta_b^{-i}(v)=v$, which is a contradiction.
Hence $G^0=\{1\}$ and $G$ is finite.
\end{proof}

\begin{cor}\label{finiteaut}
Let $V$ be a VOA. 
Assume that $V$ is rational or $C_2$-cofinite.
Moreover, we assume that $\Aut V$ contains an elementary abelian $p$-group $A$ of order  $p^2$ such that $C_{\Aut V}(a)$ is finite for any $a\in A\setminus \{1\}$. 
Then the full automorphism group $\Aut V$ is finite.
\end{cor}

\begin{proof}
Since $V$ is $C_2$-cofinite or rational, $V$ is finitely generated. Thus, by \cite{DG}, $\Aut V$ is an algebraic group.  The conclusion now follows from Theorem \ref{finiteness}.
\end{proof}

Combining Lemma \ref{Lem:finite} and Corollary \ref{finiteaut}, we obtain the following theorem:

\begin{thm}\label{Autfinite}
The group $\Aut V^\sharp$ is finite.
\end{thm}

Let us consider a Sylow $3$-subgroup of $\Aut V$.
Recall from Proposition \ref{prop:2.6} that $H_1= N_{\Aut V^\sharp} (\langle \tau' \rangle  )  \cong  3^{ 1 +12}. ( 2. \Suz {:}2)$ and that $Z(O_3(H_1))=\langle\tau'\rangle$.

\begin{lem}\label{syl3}
The subgroup $H_1$ contains a Sylow $3$-subgroup of $\Aut V^\sharp$, and a Sylow $3$-subgroup of $\Aut V^\sharp$ has order $3^{20}$.
\end{lem}

\begin{proof}
Let $P$ be a Sylow $3$-subgroup of $H_1$, and let $\tilde{P}$ be a Sylow $3$-subgroup of $\Aut V^\sharp$ such that $P \subset \tilde{P}$.
Let us show that $Z(\tilde{P}) =Z(O_3(H_1))$.
Since  $P \subset \tilde{P}$, we have $1\neq Z(\tilde{P}) \subset Z(P)$.
Recall from Proposition \ref{prop:2.6} that $C_{H_1}(O_3(H_1))=Z(O_3(H_1))$.
It follows from $O_3(H_1)\subset P$ that $Z(P)\subset C_{H_1}(O_3(H_1))=Z(O_3(H_1))$.
Clearly, $Z(O_3(H_1))\subset Z(P)$, and $Z(P)=Z(O_3(H_1))\cong \Z_3$.
Hence $Z(P)=Z(\tilde{P})$.
Thus $\tilde{P}\subset H_1$, and $P=\tilde{P}$.

Now note that the sporadic Suzuki group $\Suz$ has order $2^{13}\cdot3^7\cdot 5^2\cdot 7\cdot 11\cdot 13$ and hence a Sylow $3$-subgroup of $H_1\cong 3^{ 1 +12}. ( 2. \Suz {:}2)$ has order $3^{20}$.
\end{proof}

\begin{thm}\label{monster}
The full automorphism group of $V^\sharp$ is isomorphic to the Monster group.
\end{thm}

\begin{proof}
By Theorem \ref{Autfinite} and Lemma \ref{syl3}, $\Aut V^\sharp$ is finite and a Sylow $3$-subgroup
has order $3^{20}$. By \cite[p52]{Atlas}, we have
$| PSU(4,3)|= 2^7\cdot 3^6\cdot 5\cdot 7$, and hence $H_1\cap H_2$ contains a Sylow $3$-subgroup of $\Aut V^\sharp$ by Lemmas \ref{Lem:inter} and \ref{syl3}.
Clearly, $H_1$ and $H_2$ contain a common Sylow $3$-subgroup of $\Aut V^\sharp$.
Moreover, by Propositions \ref{prop:2.6} and \ref{NA} and Lemma \ref{Lem:inter}, the subgroups $H_1$ and $H_2$ satisfy the hypotheses (i), (ii) and (iii) in Theorem \ref{thm:3local_monster}. Thus, we have $\Aut V^\sharp \cong \M$.
\end{proof}

\section{Unitary VOA}

\subsection{Unitary VOA and unitary modules}

In this subsection, we recall the notion of unitary VOAs and unitary modules from \cite{Dlin}.

\begin{defi} Let $(V, Y, \1,\omega)$ be a vertex operator algebra. An anti-linear
isomorphism $\phi : V \to V$ is called an {\it anti-linear automorphism} of $V$ if $\phi(\1) = \1, \phi(\omega) =\omega$ and $\phi(u_n v) = \phi(u)_n \phi(v)$ for any $u, v \in  V$ and $n\in \Z$.
An {\it anti-linear involution} of $V$ is an anti-linear automorphism of order $2$.
\end{defi}

\begin{defi}[\cite{Dlin}] Let $(V, Y, \1,\omega)$ be a vertex operator algebra and let $\phi:V\to V$ be an anti-linear involution of $V$. Then $(V, \phi)$ is said to be {\it unitary} if there exists a positive-definite Hermitian form $(\ ,\ )_V : V \times V \to \C$, which is $\C$-linear on the first vector and anti-$\C$-linear on the second vector, such that the following invariant property holds for any $a, u, v\in V$:
\[
(Y (e^{zL(1)} (-z^{-2})^{L(0)} a, z^{-1})u, v)_V = (u, Y (\phi(a), z)v)_V,
\]
where $L(n)$ is defined by $Y (\omega, z) = \sum_{n\in \Z} L(n)z^{-n-2}$.
\end{defi}

\begin{rmk}\label{Rem:inv}
Let $(V,\phi)$ be a simple unitary VOA with invariant Hermitian form $(\cdot,\cdot)_V$. 
Then $V$ is self-dual and of CFT-type (\cite[Proposition 5.3]{CKLW}), and $V$ has a unique invariant symmetric bilinear form $\langle\cdot,\cdot\rangle$ up to scalar (\cite{L94}).
Normalizing $(\1,\1)_V=\langle\1,\1\rangle=1$, we obtain $(u,v)_V=\langle u,\phi(v)\rangle$ for all $u,v\in V$.
Note that $(\phi(u),\phi(v))_V=\overline{(u,v)}_V=(v,u)_V$ for $u,v\in V$.
\end{rmk}
\begin{defi}[\cite{Dlin}]
Let $(V,\phi)$ be a unitary VOA and $g$ a finite order automorphism of $V$. 
An (ordinary) $g$-twisted $V$-module $(M, Y_M)$ is called a {\it unitary $g$-twisted $V$-module} if there exists
a positive-definite Hermitian form $(\ ,\  )_M : M \times M \to  \C$ such that the following
invariant property holds for $a\in V$ and $w_1,w_2 \in M$:
\begin{equation}
(Y_M(e^{zL(1)} (-z^{-2})^{L(0)} a, z^{-1})w_1, w_2)_M = (w_1, Y_M(\phi(a), z)w_2)_M.\label{Eq:Inv}
\end{equation}
We call such a form {\it a positive-definite invariant Hermitian form}.
\end{defi}

The following lemma follows from the similar argument as in \cite[Remark 5.3.3]{FHL}.

\begin{lem}[{cf.\ \cite[Remark 5.3.3]{FHL}}]\label{Rem:specify}
Let $(V,\phi)$ be a unitary VOA.
Let $M$ be a $V$-module and let $M'$ be the contregredient module of $M$ with a natural pairing $\langle\cdot,\cdot\rangle$ between $M$ and $M'$. 
\begin{enumerate}
\item If $M$ has a non-degenerate invariant sesquilinear form $(\cdot,\cdot)$, which is linear on the first vector and anti-$\C$-linear on the second vector and satisfies the invariant property \eqref{Eq:Inv},  then the map $\Phi:M\to M'$ defined by $(u,v)=\langle u,\Phi(v)\rangle$, $u,v\in M$, is an anti-linear bijective map and $\Phi(a_nu)=\phi(a)_n\Phi(u)$ for $a\in V$ and $u\in M$.
\item If there exists an anti-linear bijective map $\Phi:M\to M'$ such that $\Phi(a_nu)=\phi(a)_n\Phi(u)$ for $a\in V$ and $u\in M$, then $(u,v)=\langle u,\Phi(v)\rangle$, $u,v\in M$, is a non-degenerate invariant sesquilinear form on $M$.
\end{enumerate}
\end{lem}

\begin{lem}\label{uniqueF}
Let $(V,\phi)$ be a unitary VOA. Let $M$ be an irreducible $V$-module. 
Then there exists at most one non-degenerate invariant sesquilinear form on $M$ (up to scalar).
\end{lem}

\begin{proof}
By Lemma \ref{Rem:specify}, a non-degenerate invariant sesquilinear form on $M$ is uniquely determined by an anti-linear bijective map $\Phi: M \to M'$ such that $\Phi(u_nw)=\phi(u)_n \Phi(w)$ for any $u\in V$ and $w\in M$.   

Let $\Phi_1$ and $\Phi_2$ be two such maps. Then it is straightforward to show that 
$\Phi_2^{-1}\Phi_1$ is a $\C$-linear map and is a $V$-module isomorphism from $M$ to $M$. Then, by Schur's lemma, $\Phi_2^{-1}\Phi_1=\lambda \mathrm{id}_M$ for some scalar $\lambda$ and we have the desired conclusion.  
\end{proof}

\begin{prop}{\rm (\cite[Theorem 5.21]{CKLW})}\label{Prop:CKLW} Let $(V,\phi)$ be a unitary VOA.
Assume that $\Aut V$ is compact.
Then the following hold:
\begin{enumerate}[{\rm (1)}]
\item Any automorphism of $V$ commutes with $\phi$.
\item A positive-definite invariant Hermitian form on $V$ is preserved by any automorphism of $V$.
\end{enumerate}
\end{prop}

\subsection{Unitary VOA structure on $\Z_3^2$-graded VOA}
In this subsection, we prove the following theorem about a unitary VOA structure on a $\Z_3^2$-graded VOA, which will be applied to the holomorphic VOA $V^\sharp$, later.

\begin{thm}\label{Thm:uni} Let $V$ be a self-dual, simple VOA of CFT-type.
Assume that $V$ has two commuting automorphisms $f$ and $g$ of order $3$.
For $i,j\in\Z$, set $V^{i,j}=\{v\in V\mid f(v)=\xi^i v,\ g(v)=\xi^j v\}$, where $\xi=\exp(2\pi\sqrt{-1}/3)$.
Set $V^i=\bigoplus_{j=0}^2 V^{i,j}$.
Assume the following:
\begin{enumerate}[{\rm (A)}]
\item There exists an anti-linear involution $\phi$ of $V^0$ such that $(V^0,\phi)$ is a unitary VOA;
\item For $i\in\{1,2\}$, $V^i$ is a unitary $(V^0,\phi)$-module;
\item There exists an automorphism $\psi$ of $V$ such that $\psi^{-1}f\psi=g$ and $\psi^{-1}g\psi=f$;
\item Both the full automorphism groups of $V^{0,0}$ and $V^0$ are compact;
\end{enumerate}
Then there exist an anti-linear involution $\Phi$ of $V$ such that $(V,\Phi)$ is a unitary VOA.
\end{thm}

\begin{rmk}\label{Rem:3^2}
\begin{enumerate}[(1)]
\item For $i,j,k,\ell\in\Z$, we have $V^{i,j}=V^{i+3k,j+3\ell}$.
\item Let $i,j,k,\ell\in\Z$.
Since the $V^{0,0}$-module $V^{i,j}$ is irreducible and $V=\bigoplus_{0\le i,j\le 2}V^{i,j}$ is  $\Z_3^2$-graded, we have $\Span_\C\{ u_nv\mid u\in V^{i,j},\ v\in V^{k,\ell},\ n\in\Z\}=V^{i+k,j+\ell}$.
\item By the $\Z_3$-grading of $V$, the contragredient module of $V^i$ is $V^{2i}$.
Let $\langle\cdot,\cdot\rangle$ be the symmetric invariant bilinear form of $V$ such that $\langle \1,\1\rangle=1$.
Then this form gives a natural pairing between $V^1$ and $V^2$.
Moreover, since $V$ is of CFT-type, for any automorphism $h$ of $V$, we have $\langle u,v\rangle=\langle h(u),h(v)\rangle$ for $u,v\in V$.
\item By the $\Z_3^2$-grading of $V$, the contragredient module of $V^{i,j}$ is isomorphic to $V^{2i,2j}$.
\end{enumerate}
\end{rmk}

Let $V$ be a VOA satisfying the assumptions of the theorem above.
Let $(\cdot,\cdot)_{V^0}$ be the positive-definite invariant Hermitian form on $V^0$ normalized so that $(\1,\1)_{V^0}=1$.
Let $\langle\cdot,\cdot\rangle$ be the normalized symmetric invariant bilinear form on $V$ such that $\langle\1,\1\rangle=1$.
Note that $(u,v)_{V^0}=\langle u,\phi(v)\rangle$ for $u,v\in V^{0}$ by Remark \ref{Rem:inv}.
By the assumption (C), $\psi(V^0)=V^{0,0}\oplus V^{1,0}\oplus V^{2,0}$ is also a unitary VOA with the anti-linear automorphism $\psi\phi \psi^{-1}$ and a positive-definite invariant Hermitian form 
\begin{equation}
(a,b)_{\psi(V^0)}=(\psi^{-1}(a),\psi^{-1}(b))_{V^0}\quad \text{for}\quad a,b\in \psi(V^0).\label{Eq:p()}
\end{equation}
Note that $\psi\phi\psi^{-1}=\phi$ on $V^{0,0}$ by the assumption (D) and Proposition \ref{Prop:CKLW} (2).

Let $i\in\{1,2\}$.
By Lemma \ref{uniqueF}, a positive-definite invariant Hermitian form on the unitary $(V^{0,0},\phi)$-module $V^{i,0}$ is unique up to scalar.
We may choose a positive-definite invariant Hermitian form $(\cdot,\cdot)_{V^i}$ on $V^i$ so that 
\begin{equation}
(u,v)_{V^i}=(u,v)_{\psi(V^0)}\quad \text{for}\quad u,v\in V^{i,0}.\label{Eq:choice}
\end{equation}
By Lemma \ref{Rem:specify} and Remark \ref{Rem:3^2} (3), there exists an anti-linear bijective map $\Phi^i:V^i\to V^{2i}$ such that $\Phi^i(a_nv)=\phi(a)_n\Phi^i(v)$ for $a\in V^0$ and $v\in V^i$ and 
\begin{equation}
(u,v)_{V^i}=\langle u,\Phi^i(v)\rangle\quad \text{for}\quad u,v\in V^i.\label{Eq:()<>}
\end{equation}
By \eqref{Eq:p()}, \eqref{Eq:choice} and \eqref{Eq:()<>}, for  any $u,v\in V^{i,0}$, we have 
$$\langle u,\Phi^i(v)\rangle=(u,v)_{V^i}=(\psi^{-1}(u),\psi^{-1}(v))_{V^0}=\langle \psi^{-1}(u),\phi\psi^{-1}(v)\rangle=\langle u,\psi\phi\psi^{-1}(v)\rangle.$$
Hence
\begin{equation}
\psi\phi\psi^{-1}=\Phi^i\quad \text{on}\quad V^{i,0}.\label{Eq:comm}
\end{equation}
Consider $\psi^2$ as an automorphism of $V^0$.
Then $\psi^2$ commutes with $\phi$ by Proposition \ref{Prop:CKLW} (1) and the assumption (D).
Hence, we have
\begin{equation}
\psi\Phi^i\psi^{-1}=\phi\quad \text{on}\quad V^{0,i}.\label{Eq:comm2}
\end{equation}
Since the order of $\phi$ is $2$, both the composition maps $\Phi^{2i}\circ\Phi^i$ and  $\Phi^{i}\circ\Phi^{2i}$ are the identity map on $V^{i,0}$ by \eqref{Eq:comm}.
Viewing $V^{2i}$ as an irreducible unitary $(V^0,\phi)$-module, we have $\Phi^{2i}=(\Phi^i)^{-1}$ on $V^{2i}$ by the same argument as in the proof of Lemma \ref{uniqueF}.

Now, we define the anti-linear map $\Phi:V\to V$ so that 
\begin{align*}
\Phi(u) =\begin{cases} \phi(u)& \text{for}\ u\in V^0\\
\Phi^i(u)& \text{for}\ u\in V^i,\ i=1,2
\end{cases}
\end{align*}
and the positive-definite Hermitian form $(\cdot,\cdot)$ on $V$ by
\begin{align*}
(u,v)=\begin{cases}
(u,v)_{V^i} & \text{ if } u,v\in V^i,\ i=0,1,2,\\
0 & \text{ if } u\in V^i,\ v\in V^j,\ i\neq j.
\end{cases}
\end{align*}
Clearly, $\Phi$ is bijective and $\Phi\circ \Phi$ is the identity of $V$.
Let us show that $(V,\Phi)$ is unitary. 

\begin{lem}\label{Lem:<>()}
\begin{enumerate}[{\rm (1)}]
\item For $i,j\in\{0,1,2\}$, $\Phi(V^{i,j})=V^{2i,2j}$.
\item For $u,v\in V$, $(u,v)=\langle u,\Phi(v)\rangle$.
\end{enumerate}
\end{lem}
\begin{proof} Clearly, $(V^{0,0},\phi)$ is a simple unitary VOA and $V^{i,j}$ is an irreducible unitary $(V^{0,0},\phi)$-module.
Hence by Lemma \ref{Rem:specify} (2) and Lemma \ref{uniqueF}, the map $\Phi$ sends $V^{i,j}$ to the submodule of $V^{2i,2j}$ isomorphic to the contragredient module of $V^{i,j}$, which is $V^{2i,2j}$ by Remark \ref{Rem:3^2} (4).
Hence we obtain (1).

Let $u\in V^i$, $v\in V^j$ with $i\neq j$.
Clearly $i+2j\neq0$.
By (1), we have $\Phi(v)\in V^{2j}$.
Since the contragredient module of $V^i$ is not isomorphic to $V^{2j}$, we have $\langle u,\Phi(v)\rangle=0$.
By \eqref{Eq:()<>} and the definition of the form $(\cdot,\cdot)$, we obtain (2).
\end{proof}

\begin{prop}\label{Prop:anti} The anti-linear map $\Phi$ is an anti-linear involution of $V$.
\end{prop}
\begin{proof} Since $\phi$ is an anti-linear automorphism of $V^0$, $\Phi$ fixes the vacuum vector and the conformal vector of $V$.
Since $(V^0,\phi)$ is unitary, the equation 
\begin{equation}
\Phi(u_nv)=\Phi(u)_n\Phi(v)\label{Eq:anti}
\end{equation} holds for $u,v\in V^0$ and $n\in\Z$.
By the definition of $\Phi^i$ for $i=1,2$, \eqref{Eq:anti} holds for $u\in V^0$ and $v\in V^i$.
By the skew symmetry, we have 
\[
u_{n}v = (-1)^{n+1} v_{n}u +\sum_{i\geq 1} \frac{(-1)^{n+i+1}}{i!} L(-1)^i (v_{n+1} u)
\]
for $u,v\in V$ and $n\in\Z$.
Hence the equation \eqref{Eq:anti} also holds for $u\in V^i$ and $v\in V^0$.

Let $x\in V^{0,j}$, $y\in V^{i,0}$ and $u\in V^{k,\ell}$.
By Borcherds' identity, for $r,q\in\Z$,
$$(x_{r}y)_{q}u=\sum_{i=0}^\infty(-1)^i\binom{r}{i}\left(x_{r-i}(y_{q+i}u)-(-1)^ry_{q+r-i}(x_{i}u)\right).$$
By the assumptions on $x$ and $y$ and the identity above, 
we have $$\Phi((x_{r}y)_{q}u)=(\Phi(x)_{r}\Phi(y))_{q}\Phi(u)=\Phi(x_{r}y)_{q}\Phi(u).$$
By Remark \ref{Rem:3^2} (2), we obtain $\Phi(u_{n}v)=\Phi(u)_{n}\Phi(v)$ for all $x,y\in V$ and $n\in\Z$.
\end{proof}

By Lemma \ref{Lem:<>()} (2), the invariant property of $\langle\cdot,\cdot\rangle$ and Proposition \ref{Prop:anti}, we obtain the following proposition:

\begin{prop}\label{Prop:Unitary}
The positive-definite Hermitian form $(\ , \ )$ on $V$ satisfies the invariant property for $(V,\Phi)$.
\end{prop} 



Combining Propositions \ref{Prop:anti} and \ref{Prop:Unitary}, we have proved Theorem \ref{Thm:uni}.

\section{Unitary form on $V^\sharp$}\label{sec:4}

In this section,  we will show that $V^\sharp$ is unitary. 

\subsection{Unitary form on $V_L^\tau$}

First, we recall from \cite{Dlin} a unitary form for a general lattice VOA $V_L$.

Let $L$ be a positive-definite even lattice. Let $V_L = M(1)\otimes_\C \C\{L\}$ be the lattice VOA as defined in \cite{FLM}.
There exists a positive-definite Hermitian form on $\C\{L\}=\Span_{\C}\{e^\alpha\mid \alpha\in L\}$ determined by $(e^\alpha, e^\beta)=\delta_{\alpha, \beta}$.
In addition, there exists a positive-definite Hermitian form on $M(1)=\Span_{\C}\{\alpha_1(-n_1)\dots\alpha_k(-n_k)\1\mid \alpha_i\in L,\ n_i\in\Z_{>0}\}$ such that 
\begin{align*}
(\1,\1)&=1;\\
(\alpha(n) u, v)&= (u, \alpha(-n) v), \quad \text{ for }\alpha \in L
\end{align*}
for any $u,v\in M(1)$.
Then a positive-definite Hermitian form on $V_{L}$ is defined by $
(u\otimes e^\alpha, v\otimes e^\beta)= (u,v)\cdot (e^\alpha, e^\beta)$, where $u,v\in M(1)$ and $\alpha,\beta\in L$.

Let $\phi: V_L\to V_L$ be an anti-linear map determined by:
\[
\alpha_1(-n_1)\cdots \alpha_k(-n_k)\otimes e^\alpha \mapsto (-1)^k \alpha_1(-n_1)\cdots \alpha_k(-n_k)\otimes e^{-\alpha},
\]
where $\alpha_1, \dots, \alpha_k, \alpha\in L$. 

\begin{thm}{\rm (\cite[Theorem 4.12]{Dlin})} Let $L$ be a positive-definite even lattice and $\phi$ be the anti-linear
map of $V_L$ defined above. Then the lattice vertex operator algebra $(V_L , \phi)$ is
a unitary vertex operator algebra.
\end{thm}

Now, we assume that $L$ has a fixed-point free isometry $\tau$ of order $3$.
We often denote a lift of $\tau$ on $V_L$ by $\tau$, also.
%
%
%
One can directly check that $\tau$ commutes with $\phi$.
Hence by \cite[Corollary 2.7]{Dlin}, we obtain the following lemma:

\begin{lem}\label{UVLtau}
Let $V_L^\tau$ be the fixed-point subspace of $\tau$ on $V_L$. Then 
we have $\phi(V_L^\tau) \subset V_L^\tau$. In particular, $(V_L^\tau, \phi)$ is a unitary VOA. 
\end{lem}




\subsection{Unitary form on $\tau^i$-twisted $V_L$-modules} \label{sec:5.2}
Let $L$ be an even positive-definite lattice with a $\Z$-bilinear form $\la\cdot|\cdot \ra$.
Assume that $L$ has a fixed-point free isometry $\tau$ of order $3$.
Next we review the construction of $\tau^i$-twisted $V_L$-modules from \cite{Dl96,TY13}. 

Define $\mathfrak{h}=\C\otimes_{\Z}L$ and
extend the $\Z$-form $\la\cdot|\cdot \ra $ $\C$-linearly to $\h .$
Denote
$$\h_{(n)}=\{ \alpha\in
\h\,|\, \tau\alpha= \xi^n \alpha \} \quad \text{for } n\in \Z, $$ where
$\xi=\exp({2\pi \sqrt{-1}/3})$.

Let $\hat{\h}[\tau]=\coprod_{n\in \Z}\h_{(n)}\otimes t^{n/3}\oplus \C c$ be 
the  $\tau$-twisted affine Lie algebra of $\h$. 
Denote
\[
\hat{\h}[\tau]^+=\coprod_{n>0}\h_{(n)}\otimes t^{n/3},\quad
\hat{\h}[\tau]^-=\coprod_{n<0}\h_{(n)}\otimes t^{n/3},\quad
\text{and} \quad \hat{\h}[\tau]^0 =\h_{(0)}\otimes t^{0}\oplus \C c,
\]
and form an induced module
\[
S[\tau]=U(\hat{\h}[\tau])\otimes_{U(\hat{\h}[\tau]^+\oplus \hat{\h}[\tau]^0)}
\C \cong S(\hat{\h}[\tau]^-) \quad \text{(linearly),}
\]
where $\coprod_{n>0}\h_{(n)}\otimes t^{n/3}$ acts trivially on
$\C$ and $c$ acts as $1$, and $U(\cdot)$ and $S(\cdot)$ denote the
universal enveloping algebra and symmetric algebra, respectively. For any $\alpha\in L$ and $n\in \frac{1}3 \Z$, let $\alpha_{(3n)}$ be the natural projection of $\alpha$ in  $\h_{(3n)}$ and we denote $\alpha(n) = \alpha_{(3n)}\otimes t^{n}$.

For any positive integer $n$, let $\la \kappa_n\ra $ be a cyclic group of order $n$.  Define $c^\tau:L\times L
\to \la \kappa_6\ra $ by $
c^\tau(\alpha,\beta)=\kappa_3^{\la \tau \alpha|\beta\ra- \la \tau^2 \alpha|\beta\ra}$ and consider the central extension
\[
1\ \longrightarrow\  \la \kappa_6\ra \ \longrightarrow\  \hat{L}_\tau \
\bar{\longrightarrow\ } L \longrightarrow\  1
\]
such that
$ aba^{-1}b^{-1}=c^\tau(\bar{a},\bar{b})$ 
for  $a,b\in \hat{L}_\tau $, where $\kappa_3=\kappa_6^2$.
Let $\hat{\tau}$ be a lift of $\tau$ in $\Aut \hat{L}_\tau$ such that $\hat{\tau}(a)=a$ for $a\in \hat{L}_\tau$ with $\tau(\bar{a})=\bar{a}$.  
By abuse of notation, we often denote $\hat{\tau}$ by $\tau$. 

\begin{rmk}\label{Rem:Setid}
By \cite[Remark 2.1]{Dl96}, there is a set-theoretic identification between the central extensions $\hat{L}$ and $\hat{L}_\tau$ such
that the respective group multiplications $\times$ and $\times_\tau$
are related by
\begin{equation*}
a\times b= \kappa_6^{\la\tau\bar{a}| \bar{b}\ra} a\times_\tau b.
\end{equation*}
\end{rmk}

Let $T$ be an irreducible $\hat{L}_\tau$-module on which  $K=\{a^{-1}\hat{\tau}(a)\mid a\in
\hat{L}_\tau\}$ acts trivially and $\kappa_6$ acts as
multiplication by $\exp(2\pi\sqrt{-1}/6)$ (cf. \cite[Remark 4.2]{Dl96}). Then the twisted space $V_L^T(\tau )=S[\tau ]\otimes T$
forms an irreducible $\tau$-twisted $V_L$-module with the vertex operator $Y^\tau(\cdot, z): V_L \to
\mathrm{End}(V_L^T) [[z,z^{-1}]]$ on $V_L^T$ defined as follows: 
For
$a\in\hat{L}$, define
\[
W^\tau(a,z)=3^{-\la\bar{a}|\bar{a}\ra/2} \sigma(\bar{a})
E^-(-\bar{a},z)E^{+}(-\bar{a},z) a z^{-\la\bar{a}|\bar{a}\ra/2},
\] 
where
\[ 
E^{\pm}(\alpha,z)=\exp\left(\sum_{n\in \frac{1}3\mathbb{Z}^{\pm}}
\frac{\alpha(n)}{n}z^{-n}\right)
\quad \text{ and } \quad \sigma(\alpha)= (1-{\xi}^2)^{\langle \tau \alpha|\alpha\rangle}.
\]
Note that $a\in\hat{L}$ acts on $T$ as the element of $\hat{L}_\tau$ via the identification in Remark \ref{Rem:Setid}.

For $\alpha_1,\dots,\alpha_k \in\mathfrak{ h}$, $n_1,\dots,n_k>0$,
and $v=\alpha_1(-n_1)\cdots\alpha_k(-n_k)\cdot\iota(a)\in V_L,$
set
\begin{equation*}
W(v,z)=\nop 
\left(\frac{1}{(n_1-1)!}\left(\frac{d}{dz}\right)^{n_1-1}
\alpha_1(z)\right)\cdot\cdot\cdot\left(\frac{1}{(n_k-1)!}
\left(\frac{d}{dz}\right)^{n_k-1} \alpha_k(z)\right)W^{\tau}(a,z)
\nop,
\end{equation*}
where $\alpha(z)=\sum_{n\in \Z/3} \alpha(n)z^{-n-1}$ and $\nop\cdots \nop$ denotes the normal ordered product.

Define constants $c_{mn}^i\in\mathbb{ C}$ for $m, n\ge 0$ and
$i=0,\cdots, 2$ by the formulas
\begin{gather}\label{cmn}
\sum_{m,n\ge 0}c_{mn}^0x^my^n=-\frac{1}{2}\sum_{r=1}^{2}{\rm
log}\left(\frac
{(1+x)^{1/3}-\xi^{-r}(1+y)^{1/3}}{1-\xi^{-r}}\right),\\
\sum_{m,n\ge 0}c_{mn}^ix^my^n=\frac{1}{2}{\rm log}\left( \frac
{(1+x)^{1/3}-\xi^{-i}(1+y)^{1/3}}{1-\xi^{-i}}\right)\ \text{
for}\ \ i\ne0.
\end{gather}

Let $\{\beta_1,\cdot\cdot \cdot, \beta_d\}$ be an orthonormal
basis of $\mathfrak{h}$ and set
\begin{equation}
\Delta_z=\sum_{m,n\ge 0}\displaystyle{\sum^{2}_{i=0}}\
\displaystyle{ \sum^d_{j=1}}
c_{mn}^i(\tau^{-i}\beta_j)(m)\beta_j(n)z^{-m-n}.
\end{equation}
Then $e^{\Delta_z}$ is well-defined on $V_L$ since $c_{00}^i=0$
for all $i$, and for $v\in V_L,$ $e^{\Delta_z}v\in V_L[z^{-1}].$
Note that $\Delta_z$ is independent of the choice of orthonormal
basis and
\begin{equation*} \hat{\tau}\Delta_z=\Delta_z\hat{\tau} \qquad
\text{and} \qquad \hat{\tau}e^{\Delta_z}=e^{\Delta_z}\hat{\tau}\quad
\text{ on } V_L.
\end{equation*}
For $v\in V_L,$  the vertex operator $Y^{\tau}(v,z)$ is defined by
\begin{equation}\label{dyg}
Y^{\tau}(v,z)=W(e^{\Delta_z}v,z).
\end{equation}

Recall that the irreducible  $\hat{L}_\tau$-module can be constructed as follows: 
Let $\mathcal{A}> K$ be a maximal abelian subgroup of $\hat{L}_\tau$. 
Let $\chi: \mathcal{A}/K\to \C$ be a linear character of $\mathcal{A}/K$. Let $\C_\chi$ be the $1$-dimensional module  of $\mathcal{A}$ affording $\chi$. Then we obtain the irreducible $\hat{L}_\tau$-module
\[
T_\chi=\mathrm{Ind}_\mathcal{A}^{\hat{L}_\tau} \C_\chi = \C[\hat{L}_\tau]\otimes_{\C[\mathcal{A}]} \C_\chi.
\]

Next we will define a Hermitian form on $V_L^{T_\chi}(\tau)$.  For any $a, b\in \hat{L}_\tau$, define 
\begin{equation}\label{ta}
(t(a), t(b)) = 
\begin{cases}
0 & \text{ if }  a\mathcal{A} \neq b\mathcal{A}, \\
\chi(b^{-1}a) & \text{ if } a\mathcal{A} = b\mathcal{A},
\end{cases}
\end{equation}
where $t(a) = a \otimes 1 \in T_\chi$ for $a\in \hat{L}$. Also using the same argument as in \cite{FLM,KR}, there
is a positive-definite Hermitian form $(\ ,\ )$ on $S[\tau]$ such that
\[
\begin{split}
(1, 1) &= 1,\\
(\alpha(n) \cdot  u, v) & = (u, \alpha(-n)\cdot v),
\end{split}
\]
for any $u, v \in S[\tau]$ and $\alpha\in L$. Now we define a positive-definite
Hermitian form on $V_L^{T_\chi}(\tau)$ by 
\[
(u\otimes r, v\otimes s)= (u,v)\cdot (r,s), \quad \text{ where } u,v\in S[\tau], r,s\in T_\chi.
\] 

\begin{lem}\label{ealpha}
For any $u,v\in V_L^{T_\chi}(\tau)$ and $a\in \hat{L}_\tau$, we have $(a\cdot u, a\cdot v)= (u,v)$.
\end{lem}

\begin{proof}
It suffices to consider the case for 
\[
u=v_1\otimes t(b_1)\quad \text{ and } \quad v=v_2\otimes t(b_2), 
\]
where 
$ v_1, v_2\in S[\tau]$ and $a,b\in \hat{L}_\tau$.
By definition, we have $$(a\cdot u, a\cdot v)= (v_1\otimes t(ab_1), v_2\otimes t(ab_2))=
(v_1, v_2)\cdot (t(ab_1), t(ab_2)).$$
Moreover, $(ab_2)^{-1} a b_1 = b_2^{-1} a^{-1} a b_1=b_2^{-1}b_1$. Therefore, we have 
$\chi((ab_2)^{-1}ab_1) =  \chi(b_2^{-1}b_1)$
if $b_1\mathcal{A}=b_2\mathcal{A}$. Hence, we  have 
$(a\cdot u, a\cdot v)= (u,v)$ as desired. 
\end{proof}

\begin{lem}
For any $\alpha\in L$ and $u,v\in V_L^{T_\chi}(\tau)$, we have 
\[
(e^\alpha \cdot u, v)= (u, \xi^{\la \alpha| \alpha \ra/2} e^{-\alpha}v). 
\]
\end{lem}

\begin{proof}
Recall from Remark \ref{Rem:Setid} that there exists a set-theoretic identification between $\hat{L}$ and $\hat{L}_\tau$.
Since $\tau$ is fixed-point free and of order $3$, we have 
$\la \tau\alpha| -\alpha\ra = \frac{1}2 \la \alpha| \alpha\ra$.  
It follows from $e^{\alpha} \times e^{-\alpha} =\kappa_2^{\la \alpha| \alpha\ra /2} e^0$ that  $e^{\alpha}\times_\tau e^{-\alpha}= \kappa_3^{-\la \alpha| \alpha\ra/2} e^0$. 
Now by Lemma \ref{ealpha}, 
\[
(u, \xi^{\la \alpha| \alpha \ra/2} e^{-\alpha}v)=  ( e^\alpha\cdot u, e^\alpha\cdot(\xi^{\la \alpha| \alpha \ra/2}  e^{-\alpha}v) ) =(e^\alpha \cdot u, v)
\]
as desired. 
\end{proof}

The following lemma is very similar to \cite[Theorem 4.14]{Dlin}. 

\begin{lem}\label{z3t}
For any $\chi$, $V_L^{T_\chi}(\tau)$ is a unitary $\tau$-twisted module of $(V_L,\phi)$. 
\end{lem}

\begin{proof}
We only need to verify the invariant property. 
Since the VOA $V_L$ is generated by $\{\alpha(-1)\cdot \1\mid \alpha\in L\} \cup \{e^\alpha \mid \alpha\in L\}$, 
it is sufficient to check
\[
(Y^\tau (e^{zL(1)} (-z^{-2} )^{L(0)} x, z^{-1})u, v) = (u, Y^\tau (\phi(x), z)v)
\]
for $x\in \{\alpha(-1)\cdot \1\mid \alpha\in L\} \cup \{e^\alpha \mid \alpha\in L\}$ and  $u,v \in V_L^{T_\chi}(\tau)$ (cf.\ \cite[Proposition 2.12]{Dlin}).

Let  $u = v_1\otimes t(a)$ and $v = v_2 \otimes t(b)$ for some $v_1 , v_2 \in S[\tau]$, $a, b \in \hat{L}$. Then
\[
(\alpha(n) u, v) = (u, \alpha(-n)v)
\]
for any $\alpha\in L$ and $n \in \frac{1}3\Z$. Thus for $x = \alpha(-1)\cdot \1$, we have
\[
\begin{split}
& (Y^\tau (e^{zL(1)} (-z^{-2} )^{ L(0)}\alpha (-1)\cdot \1 , z^{-1}) u, v)\\
= &-z^{-2} (Y^\tau(\alpha(-1)\cdot \1, z^{-1}) v_1 \otimes t(a), v_2\otimes t(b))\\
=& -z^{-2}\sum_{n\in \frac{1}3\Z} (\alpha(n)v_1 , v_2 )(t(a), t(b))z^{n+1}\\
= & -\sum_{n\in \frac{1}3\Z} (v_1 , \alpha(-n)v_2 )(t(a), t(b)) z^{n-1}\\
= &(u, Y^\tau (\phi(\alpha(-1)\cdot \1), z)v).
\end{split}
\]
Notice that $e^{\Delta_z}(\alpha(-1)\cdot \1) = \alpha(-1)\cdot \1$. 

Now take $x = e^\alpha$ with $(\alpha, \alpha) = 2k$. Then we have
\[
\begin{split}
& (Y^{\tau} (e^{ zL(1)} (-z^{-2} )^{L(0)} e^\alpha , z^{-1} )u, v)\\
= &(Y^{\tau}(e^{ zL(1)} (-z^{-2} )^{L(0)} e^\alpha , z^{-1} )v_1 \otimes t(a), v_2\otimes t(b))\\
= &(-z^{-2} )^k (3^{-k} (1-\xi^2)^{-k} E^-(-\alpha,z^{-1})E^{+}(-\alpha,z^{-1}) e^\alpha z^{k} v_1 \otimes t(a), v_2\otimes t(b))\\
= & (-z^{-2} )^k  (  v_1 \otimes t(a), 3^{-k} \overline{(1-\xi^2)^{-k}}E^-(\alpha,z)E^{+}(\alpha,z) \xi^{k}e^{-\alpha} z^{k}v_2\otimes t(b))\\
= & (  v_1 \otimes t(a), 3^{-k} (1-\xi^2)^{-k} E^-(\alpha,z)E^{+}(\alpha,z) e^{-\alpha} z^{-k}v_2\otimes t(b))\\
=& (  v_1 \otimes t(a), Y^\tau(\phi(e^\alpha), z) v_2\otimes t(b)) 
\end{split}
\]
as desired. 
\end{proof}
\subsection{Unitary VOA structure on $V^\sharp$}
In this subsection, we check that $V^\sharp$ satisfies the hypotheses (A)--(D) of Theorem \ref{Thm:uni}, which shows that $V^\sharp$ is unitary.

By Theorem \ref{Thm:sharp}, 
$$V^\sharp=V_\Lambda^\tau\oplus (V_\Lambda^{T_1}(\tau))_\Z\oplus(V_\Lambda^{T_2}(\tau^2))_\Z$$
is a holomorphic VOA of CFT-type.
Clearly it is simple and self-dual.
Let $b\in K_{12}^* \setminus K_{12}$ of squared norm $6$. 
By the proof of Lemma \ref{Lem:finite}, we have $\tau'=\varphi_{S^0[1]}$.
Set $g=\varphi_{S^b[0]}$.
Then $\tau'$ and $g$ are commuting automorphisms of $V^\sharp$ of order $3$.
Set $V^{i,j}=\{v\in V\mid \tau'(v)=\xi^i v,\ g(v)=\xi^j v\}$, where $\xi=\exp(2\pi\sqrt{-1}/3)$, and set $V^i=\bigoplus_{j=0}^2 V^{i,j}$.
By the definition of $\tau'$, we have $V^0=V_\Lambda^\tau$ and $V^i=(V_\Lambda^{T_i}(\tau^i))_\Z$ for $i=1,2$.

It follows from Lemmas \ref{UVLtau} and \ref{z3t} that $V^0$ is unitary and $V^\sharp$ is a unitary $V^0$-module.
Hence $V^\sharp$ satisfies the hypotheses (A) and (B).
Let $U$ be the subspace of $R=R(V_{K_{12}}^\tau)$ spanned by $S^0[1]$ and $S^b[0]$.
Then $U$ is a $2$-dimensional totally singular subspace.
The stabilizer of $U$ in $\Omega_8^-(3)$ has the shape $3^{1+8}{:} (\GL_2(3)\times \Omega_4^-(3))$ (see \cite[p141]{Atlas}) and it acts on $U$ as $\GL(U)\cong \GL_2(3)$.
Hence, there exists $\psi\in \Omega_8^-(3)$ such that $\psi(S^0[1])=S^b[0]$ and $\psi(S^b[0])=S^0[1]$.
Remark that $f,g\in R^*$ and $N_{\Aut V^\sharp}(R^*)/R^*\cong \Omega^-_8(3).2$ acts naturally on $R^*$ by conjugation.
Hence a lift $\tilde\psi$ of $\psi$ to $N_{\Aut V^\sharp}(R^*)\cong 3^8.\Omega^-_8(3).2$ satisfies $\tilde{\psi}^{-1}f\tilde{\psi}=g$ and $\tilde\psi^{-1}g\tilde\psi=f$.
Thus $V^\sharp$ satisfies the hypothesis (C).

By \eqref{Eq:aut1}, the automorphism group of $V^0$ is finite.
Let $\bar{g}$ denote the restriction of $g$ to $V^0$.
Since $V^{0,0}$ is the $\bar{g}$-fixed points of $V^0\cong V_\Lambda^\tau$, we have $V^{0,0}=V_L^\tau$, where $L=\{v\in\Lambda\mid \langle v|b\rangle\in3\Z\}$ and $\langle\cdot|\cdot\rangle$ is the inner product of $\R\otimes_\Z\Lambda$.
By \cite{CL14}, we know that the irreducible $V_L^\tau$-modules $V^{0,1}$ and $V^{0,2}$ are simple currents.
Hence $V^{0}$ is a $\Z_3$-graded simple current extension of $V^{0,0}$.
Let $A$ be the cyclic group of order $3$ corresponding to the grading.
Then $A^*=\langle \bar{g}\rangle$.
Consider the action of $\Aut V^{0,0}$ on the set of isomorphism classes of simple current $V^{0,0}$-modules.
Let $S_A$ be the set of isomorphism classes of $V^{0,j}$, $j=0,1,2$.
By Theorem \ref{NCthm}, the stabilizer of $S_A$ is isomorphic to the finite group $N_{\Aut V^0}(\langle \bar{g}\rangle)/\langle \bar{g}\rangle$. 
Since $V^{0,0}$ is $C_2$-cofinite, there are finitely many irreducible $V^{0,0}$-modules.
Hence the $\Aut V^{0,0}$-orbit of $S_A$ is a finite set.
Thus $\Aut V^{0,0}$ is finite, and $V^\sharp$ satisfies the hypothesis (D).

The VOA $V^\sharp$ therefore satisfies all the hypotheses in Theorem \ref{Thm:uni} and we have our main result in this section.   

\begin{thm}\label{Thm:Unitary}
The VOA $V^\sharp$ has a unitary structure as an extension of unitary $V_\Lambda^\tau$-module structure.
\end{thm} 

\section{Characterization of the Moonshine VOA}

In this section, we prove the following characterization of the Moonshine VOA $V^\natural$:

\begin{thm}\label{Thm:chara} Let $V$ be a $C_2$-cofinite, holomorphic VOA of CFT type of central charge $24$. 
We also assume the following: 
\begin{enumerate}[{\rm (a)}]
\item $V_1=0$;
\item there exists an anti-linear involution $\phi$ of $V$ such that $(V,\phi)$ is unitary;
\item the automorphism group $\Aut V$ of $V$ is isomorphic to the Monster simple group $\M$; 
\item the action of $\Aut V$ on $V_2$ is not trivial;
\item $V$ contains an Ising vector fixed by $\phi$.
\end{enumerate}
Then $V$ is isomorphic to the Moonshine VOA $V^\natural$.
\end{thm}

\begin{rmk}\label{rmk:1} The Moonshine VOA $V^\natural$ satisfies the conditions (a)--(e) in the theorem above.
\end{rmk}

\begin{rmk}\label{rmk:2} The condition (c) implies that $\Aut V$ is compact.
Hence by Proposition \ref{Prop:CKLW} (1), $\Aut V$ preserves the (unique) real form $V_\R=\{v\in V\mid \phi(v)=v\}$ of $V$.
\end{rmk}

We now assume that $V$ is a $C_2$-cofinite, holomorphic VOA of CFT-type of central charge $24$ satisfying (a)--(e) in Theorem \ref{Thm:chara}.
Recall that $e\in V_2$ is an Ising vector if the subVOA $\langle e\rangle_{\rm VOA}$ generated by $e$ is isomorphic to $L(1/2,0)$ and $e$ is the conformal vector of $\langle e\rangle_{\rm VOA}$. For an Ising vector $e$, one can define an automorphism $\tau_e\in \Aut V$ and an automorphism $\sigma_e\in\Aut V^{\tau_e}$ as described in \cite{Mi96}.

\begin{lem}\label{Lem:tau} Let $e$ be an Ising vector of $V$.
Then  $\tau_e\neq 1$.
\end{lem}
\begin{proof} Suppose $\tau_e=\sigma_e=1$.
By \cite[Proposition 4.9]{Mi96}, $V$ is isomorphic to $L(1/2,0)\otimes V'$ for some VOA $V'$.
Then $V$ has an irreducible module $L(1/2,1/2)\otimes V'$ not isomorphic to $V$, which contradicts that $V$ is holomorphic.

Suppose that $\tau_e=1$ and $\sigma_e\neq1$.
Let $G$ be the subgroup of $\Aut V$ generated by $\sigma$-involutions associated to Ising vectors in $V$.
Then $G\neq\{1\}$ and $G$ is normal in $\Aut V$.
It follows from the simplicity of $\Aut V$ that $G=\Aut V$.
By \cite[Theorem 6.13]{Mi96}, $G$ is a $3$-transposition group, which contradicts that $\Aut V(\cong \M)$ is not a $3$-transposition group.
Therefore $\tau_e\neq1$.
\end{proof}

Let $I$ be the set of all Ising vectors of $V_\R$.

\begin{lem}\label{Lem:1to1}
\begin{enumerate}[{\rm (1)}]
\item For any $e\in I$, the $\tau$-involution $\tau_e$ belongs to the conjugacy class $2A$ of $\M$.
\item The map $e\mapsto \tau_e$ from $I$ to the conjugacy class $2A$ of $\M$ is bijective.
\end{enumerate}
\end{lem}
\begin{proof} 
Let $e\in I$ and set $I_e=\{f\in I\mid \tau_e=\tau_f\}$.
Let $f\in I_e$.
Recall from Lemma \ref{Lem:tau} that $V$ does not have Ising vectors of $\sigma$-type.
By \cite[Theorem 6.12]{Mi96} and $e,f\in V_\R$, along with (b), $e=f$ or $e$ is orthogonal to $f$.
Suppose $|I_e|\ge48$.
Then $V$ has a mutually orthogonal $48$ Ising vectors, and the sum of them must coincides with the conformal vector of $V$ since the invariant form on $V_\R$ is positive-definite.
Hence $V$ has a Virasoro frame whose Ising vectors induce the same $\tau$-involution.
The $1/16$-code is $\{(0^{48}),(1^{48})\}$, and the code VOA associated to the dual code of $1/16$-code, i.e., the code consisting of all even codewords, is a subVOA of $V$, which implies that $V_1\neq0$.
This contradicts (a).
Thus $|I_e|\le 47$.

Let $C$ be the centralizer of $\tau_e$ in $\Aut V$.
Recall from (c) and (d) that $V_2=\C\omega\oplus\omega^\perp$ as a direct sum of irreducible $\M$-modules, where $\omega^\perp$ is the orthogonal complement of $\omega$ with respect to the invariant form.
Suppose that $\tau_e$ belongs to the conjugacy class $2B$.
Then $V_2$ decomposes into a direct sum of irreducible $C$-modules as follows:
$$V_2=\underline{1}+\underline{299}+\underline{98280}+\underline{98304},$$
where $\underline{n}$ means an $n$-dimensional irreducible $C$-module.
Recall that $V_2$ contains a $C$-invariant subspace $\Span_\C I_e $ whose dimension is at most $47$.
This contradicts that $\underline{1}$ is spanned by the conformal vector $\omega$ and that $e\neq\omega$ and $e\in I_e$.
Hence $\tau_e$ belongs to the conjugacy class $2A$, and $V_2$ decomposes into a direct sum of irreducible $C$-modules as follows:
\begin{equation}
V_2=\underline{1}+\underline{1}+\underline{4371}+\underline{96255}+\underline{96256}.\label{Eq:C2A}
\end{equation}
Since $\Span_\C I_e $ is a $C$-submodule of $V_2$ whose dimension is at most $47$ and $\omega$ is fixed by $C$, we have $|I_e|=1$.
Hence $I_e=\{e\}$, and we obtain the injectivity.

Let $g$ be an element in the conjugacy class $2A$ of $\M$.
Then there exists $h\in\M$ such that $h^{-1}gh=\tau_e$.
Hence $g=\tau_{h(e)}$.
By Remark \ref{rmk:2}, $h(e)\in I$, and we obtain the surjectivity
\end{proof}

Theorem \ref{Thm:chara} follows from the next proposition.

\begin{prop} As algebras, $V_2$ is isomorphic to $V^\natural_2$. 
In particular,  $V$ is isomorphic to $V^\natural$ as VOAs.  
\end{prop}
\begin{proof} Let $\omega$ and $\omega^\natural$ denote the conformal vectors of $V$ and $V^\natural$, respectively.
Let $\iota_1$ be the linear map from $\C\omega$ to $\C\omega^\natural$ define by $\iota_1(\omega)=\omega^\natural$.
Let $\iota_2$ be an $\M$-module isomorphism from $\omega^\perp$ to $(\omega^{\natural})^\perp$.
Since $\omega^\perp$ and $(\omega^{\natural})^\perp$ are irreducible, $\iota_2$ is unique up to a scalar multiple.
In addition, an invariant bilinear form on $\omega^\perp$ preserved by the action of $\M$ is unique up to a scalar multiple.
We choose $\iota_2:\omega^\perp\to(\omega^{\natural})^\perp$ so that $\langle\iota_2(x),\iota_2(y)\rangle=\langle x,y\rangle$ for all $x,y\in \omega^\perp$, where $\langle\cdot,\cdot\rangle$ is the normalized invariant bilinear form on $V$.

Let us show that $\iota=\iota_1\oplus\iota_2:V_2\to V_2^\natural$ is an algebra homomorphism.
Let $e$ be an Ising vector in $(V_\R)_2$.
By \eqref{Eq:C2A}, $U=\Span_\C \{e,\omega\}$ is the fixed-point subspace of $C_\M(\tau_e)$ in $V_2$.
Since $\iota$ is a $C_\M(\tau_e)$-homomorphism, $\iota(U)$ is also the fixed-point subspace of $C_\M(\tau_e)$ in $(V^\natural_2)_\R$.
By Remark \ref{rmk:1} and Lemma \ref{Lem:1to1}, there exists an Ising vector $f\in (V^\natural_\R)_2$ such that $\iota(U)=\Span_\C\{f,\omega^\natural\}$.
Set $e'=(1/48)\omega-e$ and $f'=(1/48)\omega^\natural-f$.
Then $e'\in\omega^\perp$ and $f'\in(\omega^{\natural})^\perp$, and $\C e'$ and $\C f'$ are the fixed-point subspaces of $C_\M(\tau_e)$ in $\omega^\perp$ and $(\omega^{\natural})^\perp$, respectively.
Since $\iota$ is a $C_\M(\tau_e)$-homomorphism, we have $\iota(e')=\lambda f'$.
Since $\iota$ preserves the invariant form, we have $\lambda=\pm1$.
If $\lambda=-1$, then we replace $\iota_2$ by $-\iota_2$, and hence we may assume $\lambda=1$.
Then $\iota(e')=f'$.
Since an $\M$-invariant algebra structure on $\omega^\perp$ is unique up to a scalar, $\iota$ is an algebra isomorphism since $\iota(e'_{(1)}e')=\iota(e')_{(1)}\iota(e')$.

It follows from $V_2\cong V^\natural_2$ as algebra that $V\cong V^\natural$ as VOAs by \cite[Theorem 1]{DGL}. 
\end{proof}

Finally, let us check that $V^\sharp$ satisfies the hypotheses in Theorem \ref{Thm:chara}.

\begin{prop} The holomorphic VOA $V^\sharp$ of central charge $24$ satisfies the hypotheses (a)--(e) in Theorem \ref{Thm:chara}.
\end{prop}
\begin{proof} By Theorems \ref{Thm:sharp}, \ref{monster} and \ref{Thm:Unitary}, $V^\sharp$ satisfies (a)--(c).
By the action of subgroups of $\Aut V^\sharp$, such as $N_{\Aut V^\sharp}(\langle\tau'\rangle)$, we know that $V^\sharp$ satisfies (d).
Recall that $V_{\Lambda}^\tau$ contains a subVOA $V_{\sqrt2E_8}^\tau$.
Let $\theta\in\Aut V_{\sqrt2E_8}$ be a lift of the $-1$-isometry of $\sqrt2E_8$.
Then $\theta$ and $\tau$ commutes.
Since $V_{\sqrt2E_8}^\theta$ has exactly $496$ Ising vectors (\cite[Proposition 4.3]{LSY}) and $496$ is not a multiple of $3$, there exists an Ising vector in $V_{\sqrt2E_8}^\theta$ fixed by $\tau$.
Hence $V_{\Lambda}^\tau$ contains an Ising vector, which proves (e).
\end{proof}

Applying Theorem \ref{Thm:chara} to $V^\sharp$, we obtain Theorem \ref{Thm:main}.


\paragraph{\bf Acknowledgement.} The authors wish to thank Xingjun Lin for helpful comments on unitary vertex algebras.


\begin{thebibliography}{CCN{\etalchar{+}}85}

\bibitem[Bo86]{Bo}
R.E.\ Borcherds, \emph{Vertex algebras, Kac-Moody algebras, and the Monster}, Proc.\ Nat'l.\ Acad.\ Sci.\ U.S.A. \textbf{83} (1986), 3068--3071.
  
\bibitem[CKLW]{CKLW}
S.\ Carpi, Y.\ Kawahigashi, R.\ Longo and M.\ Weiner, \emph{From vertex operator algebras to conformal nets and back}, to appear in Mem. Amer. Math. Soc., arXiv:1503.01260

\bibitem[CCN{\etalchar{+}}85]{Atlas}
J.~H. Conway, R.~T. Curtis, S.~P. Norton, R.~A. Parker, and R.~A. Wilson, 
  \emph{Atlas of finite groups}, Oxford University Press, Eynsham, 1985,
  Maximal subgroups and ordinary characters for simple groups, With
  computational assistance from J. G. Thackray.  

\bibitem[CL16]{CL14}
H.~Chen and C.H.~Lam,  \emph{Quantum dimensions and fusion rules of the VOA $ V^\tau_\LCD$}, J. Algebra, \textbf{459} (2016), 309--349.

\bibitem[CS83]{CS83}
J.~H. Conway and N.~J.~A. Sloane, \emph{The {C}oxeter-{T}odd lattice, the
  {M}itchell group, and related sphere packings}, Math. Proc. Cambridge Philos.
  Soc. \textbf{93} (1983), 421--440. 
  
\bibitem[DGL07]{DGL} C.\ Dong, R.L.\ Griess and C.H.\ Lam, \emph{Uniqueness results for the moonshine vertex operator algebra}, Amer. J. Math. {\bf 129} (2007), 583--609.

\bibitem[DL96]{Dl96}
C.~Dong and J.~Lepowsky, \emph{The algebraic structure of relative twisted
  vertex operators}, J. Pure Appl. Algebra   \textbf{110} (1996),
  259--295.

\bibitem[DLM00]{DLM00}
C.~Dong, H.~Li, and G.~Mason, \emph{Modular-invariance of trace functions in
  orbifold theory and generalized {M}oonshine}, Comm. Math. Phys. \textbf{214}
  (2000), 1--56. 

\bibitem[DLin14]{Dlin} C.~Dong and X.J.~Lin, \emph{Unitary vertex operator algebras}, J. Algebra \textbf{397} (2014), 252--277.

\bibitem[DM94]{DM94p}
C.~Dong and G.~Mason, \emph{The construction of the moonshine module as a
  {$Z_p$}-orbifold}, Mathematical aspects of conformal and topological field
  theories and quantum groups ({S}outh {H}adley, {MA}, 1992), Contemp. Math.,
  vol. 175, Amer. Math. Soc., Providence, RI, 1994, pp.~37--52. 


\bibitem[DG02]{DG}C. Dong, R.L.  Griess, Jr., \emph{Automorphism groups and derivation algebras of finitely generated vertex operator algebras}, Michigan Math. J. {\bf 50} (2002), 227--239. 
  
\bibitem[FHL93]{FHL}
I.B.\ Frenkel, Y.\ Huang and J.\ Lepowsky, \emph{On axiomatic approaches to vertex operator algebras and modules},  Mem. Amer. Math. Soc. {\bf 104} (1993), viii+64 pp. 


\bibitem[FLM88]{FLM}
I.B.~Frenkel, J.~Lepowsky, and A.~Meurman, \emph{Vertex operator algebras and the
  monster}, Pure and Appl. Math., vol. 134, Academic Press, Boston, 1988.

\bibitem[HKL15]{HKL}Y. Z. Huang, A. Kirillov Jr.,  J. Lepowsky, \emph{Braided tensor categories and extensions of vertex operator algebras}, Commun. Math. Phys. \textbf{337} (2015),1143-1159.


\bibitem[KRR13]{KR}
V.G.\ Kac, A.K.\ Raina and N.\ Rozhkovskaya, Bombay lectures on highest weight representations of infinite dimensional Lie algebras. Second edition. Advanced Series in Mathematical Physics, {\bf 29}. World Scientific Publishing Co. Pte. Ltd., Hackensack, NJ, 2013.
  
\bibitem[KLY03]{KLY03}
M.~Kitazume, C.~H. Lam, and H.~Yamada, \emph{3-state {P}otts model, {M}oonshine
  vertex operator algebra, and {$3A$}-elements of the {M}onster group}, Int.
  Math. Res. Not. (2003), no.~23, 1269--1303. 

\bibitem[LSY07]{LSY}C.H.\ Lam, S.\ Sakuma and H.\ Yamauchi, \emph{Ising vectors and automorphism groups of commutant subalgebras related to root systems}, Math. Z. \textbf{255} (2007), no. 3, 597--626. 

\bibitem[LY14]{LY13}
C.~H. {Lam} and H.~{Yamauchi}, \emph{{On 3-transposition groups generated by
  $\sigma$-involutions associated to $c=4/5$ Virasoro vectors}}, J. Algebra, \textbf{416} (2014), 84--121.
  
\bibitem[Le85]{Lep85}
J.~Lepowsky, \emph{Calculus of twisted vertex operators} Proc. Nat. Acad. Sci. U.S.A., \textbf{82} (1985) 8295--8299
 
\bibitem[Li94]{L94}
H.~Li, \emph{Symmetric invariant bilinear forms on vertex operator algebras},
  J. Pure Appl. Algebra \textbf{96} (1994), 279--297. 
  
\bibitem[Mi96]{Mi96}
M.~Miyamoto, \emph{Griess algebras and conformal vectors in vertex operator algebras}, J. Algebra \textbf{179} (1996),  523--548.

\bibitem[Mi13]{Mi13a}
M.~Miyamoto, \emph{A {$\mathbb{Z}_3$}-orbifold theory of lattice vertex
  operator algebra and {$\mathbb{Z}_3$}-orbifold constructions}, Symmetries,
  integrable systems and representations, Springer Proceedings in Mathematics
  \& Statistics, vol.~40, Springer, Heidelberg, 2013, pp.~319--344.
   

\bibitem[SY03]{SY03}
S.~Sakuma and H.~Yamauchi, \emph{Vertex operator algebra with two {M}iyamoto
  involutions generating {$S_3$}}, J. Algebra \textbf{267} (2003), no.~1,
  272--297.  

\bibitem[SS10]{SS10}
M.~R. Salarian and G.~Stroth, \emph{An identification of the {M}onster group},
  J. Algebra \textbf{323} (2010), no.~4, 1186--1195.  

\bibitem[Sh04]{Sh04}
H.~Shimakura, \emph{The automorphism group of the vertex operator algebra $V_L^+$ for an even lattice $L$ without roots}, J. Algebra \textbf{280} (2004), 29--57.   

\bibitem[Sh07]{Sh07}
H.~Shimakura, \emph{Lifts of automorphisms of vertex operator algebras in
  simple current extensions}, Math. Z. \textbf{256} (2007), no.~3, 491--508.

\bibitem[Sh11]{Sh11}
H.\ Shimakura, \emph{An $E_8$-approach to the moonshine vertex operator algebra}, J. London Math. Soc. {\bf 83} (2011), 493--516.
  
\bibitem[Sh14]{Sh14} H.~Shimakura, \emph{The automorphism group of the $\Z_2$-orbifold of the Barnes-Wall lattice vertex operator algebra of central charge 32}, Math. Proc. Cambridge Philos. Soc. {\bf 156} (2014), 343--361.

\bibitem[TY13]{TY13}
K.~Tanabe, H.\ Yamada, \emph{Fixed point subalgebras of lattice vertex operator algebras by
  an automorphism of order three}, J. Math. Soc. Japan \textbf{65} (2013),
  1169--1242.  
  
\bibitem[Ti84]{T84} J. Tits, \emph{On R. Griess' ``friendly giant''}, Invent. Math. \textbf{78} (1984), 491--499.

\end{thebibliography}
 
\newcommand{\etalchar}[1]{$^{#1}$}
\providecommand{\bysame}{\leavevmode\hbox to3em{\hrulefill}\thinspace}
\providecommand{\MR}{\relax\ifhmode\unskip\space\fi MR }
\providecommand{\MRhref}[2]{%
  \href{http://www.ams.org/mathscinet-getitem?mr=#1}{#2}
}
\providecommand{\href}[2]{#2}

\end{document}